\documentclass[12pt]{amsart}
\usepackage[english]{babel}

\usepackage{pifont}
\usepackage{todonotes}

\usepackage{geometry}
\usepackage{amsthm,amsmath,amssymb,amsfonts,amscd,mathtools}
\usepackage[shortlabels]{enumitem}
\usepackage{hyperref,cleveref}
\usepackage{multirow,float}

\theoremstyle{plain}
\newtheorem*{maintheorem}{Main Theorem}
\newtheorem{theorem}{Theorem}[section]

\newtheorem{lemma}[theorem]{Lemma}
\newtheorem{proposition}[theorem]{Proposition}
\newtheorem{conjecture}[theorem]{Conjecture}

\newtheorem{corollary}[theorem]{Corollary}

\theoremstyle{definition}
\newtheorem{definition}[theorem]{Definition}
\newtheorem{remark}[theorem]{Remark}
\newtheorem{notation}[theorem]{Notation}

\theoremstyle{remark}
\newtheorem{example}[theorem]{Example}

\numberwithin{equation}{section}
\numberwithin{figure}{section}
\numberwithin{table}{figure}

\newcommand{\pt}[1]{\left({#1}\right)}
\newcommand{\pq}[1]{\left[{#1}\right]}

\newcommand{\rest}[2]{\left.{#1}\right|_{#2}}
\newcommand{\pg}[1]{\left\{{#1}\right\}}

\newcommand{\vol}{\operatorname{Vol}}

\newcommand{\Z}{\mathbb{Z}}

\newcommand{\R}{\mathbb{R}}

\newcommand{\C}{\mathbb{C}}

\DeclareMathOperator{\rk}{rank}

\newcommand{\g}{\mathfrak{g}}
\newcommand{\h}{\mathfrak{h}}
\newcommand{\kk}{\mathfrak{k}}
\newcommand{\q}{\mathfrak{q}}
\newcommand{\ab}{\mathfrak{a}}

\newcommand{\sll}{\mathfrak{sl}}

\title[]{$p$-Kähler structures on nilmanifolds and holomorphically parallelizable manifolds}

\author{Ettore Lo Giudice}
\address[Ettore Lo Giudice]{
Dipartimento di Scienze Matematiche, Fisiche e Informatiche
Unit\`a di Matematica e Informatica\\
Universit\`a degli Studi di Parma\\
Parco Area delle Scienze 53/A, 43124\\
Parma, Italy \& \\
Dipartimento di Matematica e Informatica \\
Universita degli Studi di Ferrara \\
Via Machiavelli 35, 44121 Ferrara, Italy}
\email{ettore.logiudice@unipr.it, ettore.logiudice@unife.it}

\author{Asia Mainenti}
\address[A. Mainenti]{}
\email{asia.mainenti@imar.ro, asia.mainenti@gmail.com}

\thanks{}

\begin{document}

\begin{abstract}
We prove the Alessandrini-Bassanelli conjecture on nilmanifolds with nilpotent complex structures and on holomorphically parallelizable solvmanifolds.
As a consequence, we classify holomorphically parallelizable solvmanifolds admitting $p$-K\"ahler structures, for lower values of $p$.
We further provide new examples of $(n-2)$-Kähler manifolds, in the setting of compact holomorphically parallelizable manifolds with reductive universal cover.
\end{abstract}

\maketitle


\section{Introduction}

This paper is set in the context of $p$-K\"ahler manifolds, namely complex manifolds endowed with $(p,p)$-forms which are closed and satisfy a pointwise positivity condition known as \textit{transversality}.
On each tangent space, this condition requires that the restriction of the form to every $p$-dimensional complex subspace is a positive volume form.
A complex manifold is $1$-K\"ahler if and only if it is K\"ahler, and if $n$ is the complex dimension, it is $(n-1)$-K\"ahler if and only if it admits a balanced metric. 
However, these are the only cases where these structures give rise to Hermitian metrics on the underlying manifold, satisfying a differential condition.

We note that recently $(n-2)$-K\"ahler structures were shown to be related to generalizations of the non-abelian Hodge correspondence, in the context of non-K\"ahler Hermitian geometry, see \cite{FM26} for more details.

\smallskip
One of the main open questions in this setting is the following:

\begin{conjecture}[Alessandrini-Bassanelli conjecture {\cite{AB}}]\label{ABconj}
    If a complex manifold $(M,J)$ admits a $p$-K\"ahler structure, then it admits $q$-K\"ahler structures for all $q\ge p.$
\end{conjecture}

Note that, should the conjecture  be true, it would restrict $p$-K\"ahler manifolds to the class of balanced manifolds.
So far the conjecture is only known to hold on a class of holomorphically parallelizable nilmanifolds (by \cite{Log2025}), and on all the explicit examples of manifolds admitting $p$-K\"ahler structures.
The main result of the paper is that the conjecture is true on some more general classes of compact quotients of Lie groups by lattices.

\begin{maintheorem}
    The Alessandrini-Bassanelli conjecture holds on $(M,J)$ satisfying one of the following:
    \begin{enumerate}
        \item\label{MTnilp} $M$ is a nilmanifold, and $J$ is nilpotent;
        \item\label{MTsolv} $(M,J)$ is a holomorphically parallelizable solvmanifold.
    \end{enumerate}
\end{maintheorem}

Whereas compact holomorphically parallelizable manifolds are always balanced, there are examples of nilmanifolds with nilpotent complex structures which are not balanced, for instance every nilmanifold admitting pluriclosed metrics.
Thus, item \ref{MTnilp} in the Main Theorem is indeed the first known evidence that the existence of $p$-K\"ahler structures does imply the existence of balanced metrics, in some special setting.

A consequence of the Main Theorem is that holomorphically parallelizable solvmanifolds cannot be $p$-K\"ahler, for $p$ less than or equal to half the complex dimension of the manifold (see Theorem \ref{n2Solv}).
Furthermore, this bound on $p$ is sharp (see Remark \ref{rmkn2}), and we can characterize which holomorphically parallelizable solvmanifolds admit $p$-K\"ahler structures for the lowest possible $p$, in Proposition \ref{lowestp}.

\smallskip
Still in the context of the Alessandrini-Bassanelli conjecture, further evidence for its validity may come from investigating the behaviour of powers of $p$-K\"ahler structures. 
Since powers of closed forms are trivially closed, the question reduces entirely to the positivity of the resulting form.
Motivated by analogous results for a weaker positivity notion, in \cite{BlockiPlis}, we prove that this problem depends on the bidegree.
Indeed, we show that the square of a transverse $(2,2)$-form is always transverse (Corollary \ref{corSquare2}), whereas there are examples of transverse $(p,p)$-forms whose square is not transverse, for every $p\ge3$ (Corollary \ref{corSquare3}).
As far as  $p$-K\"ahler structures are concerned, the latter result implies that powers of $p$-K\"ahler structures may not be $2p$-K\"ahler, for $p\ge3.$
However, the question remains open, because it is not known whether there exist complex manifolds $M$ such that the aforementioned transverse $(p,p)$-forms, whose squares are not transverse, are closed on $M$.

\smallskip
Another critical matter in the setting of $p$-K\"ahler geometry is that only a handful of examples are known, and they all arise either as modifications of compact K\"ahler manifolds, or on nilmanifolds endowed with a special class of complex structures.

A natural way to obtain new examples is by taking products of pairs of $p$-K\"ahler manifolds. However, it is not straightforward to determine whether the resulting product is $q$-K\"ahler for some $q$. In \cite{AB}, the authors addressed this problem under the assumption that the Alessandrini–Bassanelli conjecture holds on both factors, together with additional positivity conditions. 
In Theorem \ref{superprodotti}, we show that, under the assumptions that the factors are balanced and, respectively, $(n-2)$-K\"ahler and $(m-2)$-K\"ahler, where $n,m$ are the complex dimensions, then the product is $(m+m-2)$-K\"ahler, without needing the extra positivity assumptions.

We also provide a new family of examples, proving that compact quotients of reductive complex Lie groups by lattices are $(n-2)$-K\"ahler if and only if the reductive Lie group has no simple factor of rank $1$ (see Theorem \ref{P-K su reductive}), where $n$ is the complex dimension.
As noted above, these manifolds are always balanced as they are holomorphically parallelizable.
On the other hand, a similar result cannot hold in full generality on reductive Lie groups, as for instance most compact Lie groups with invariant complex structures cannot have $(n-2)$-K\"ahler structures, see \cite[Theorem 4.6]{FGM}.

\smallskip
In \cite{ABcurrents} the question of whether a $p$-K\"ahler structure can be exact is raised. The problem was studied for balanced metrics in \cite{yachou}, where the author shows that, on a compact quotient of a complex semisimple Lie group by a lattice, every invariant balanced metric is $d$-exact,
and in \cite{ouv,PiovTom}, where the authors proved that on $\mathfrak{sl}(2,\C)$ endowed with a biinvariant complex structure, the second power of the diagonal Hermitian metric has vanishing class in the Bott-Chern cohomology. Here, we show that all the aforementioned examples obtained as quotients of semisimple complex Lie groups admit a $(n-2)$-K\"ahler structure with vanishing Bott-Chern cohomology class. 

\smallskip
The paper is organized as follows. 
In Section \ref{secPrel} we recall some notions of positivity of exterior forms on complex vector spaces, and the main known properties of $p$-K\"ahler manifolds.

Section \ref{secPos} is dedicated to the study of positive forms on  complex vector spaces.
More precisely, we study the positivity of squares of transverse forms, and some further results showing how to construct transverse forms of type $(p+1,p+1)$, starting from a preferred one, of type $(p,p)$, see Lemma \ref{ambrozia}.
The latter result will play an important role in the proofs of the Alessandrini-Bassanelli conjecture.

In Section \ref{secCpxMfd} we first consider the general setting of compact complex manifold, first proving the aforementioned result on products of $p$-K\"ahler manifolds.
We later focus on compact quotients of Lie group by lattices, with a special type of left-invariant complex structures, called \textit{quasi-nilpotent}. 
We show how the existence of $p$-K\"ahler structures affects suitable quotients of the associated Lie algebras, see Proposition \ref{enrico}, and prove that in some special cases the Alessandrini-Bassanelli partially holds, see Proposition \ref{Lie algebra con particolari ipotesi}.
We conclude the section considering nilmanifolds with nilpotent complex structures, ultimately proving item \ref{MTnilp} of the Main Theorem, in Theorem \ref{polis}.

Lastly, in Section \ref{secHolPar} we specialize to compact holomorphically parallelizable manifolds.
In the solvable case, we prove item \ref{MTsolv} of the Main Theorem, see Theorem \ref{Disney upppp}, and the corollaries discussed above.
The last part of the paper treats the reductive case.

\section{Preliminaries}\label{secPrel}
In this section, we discuss the different notions of positivity for differential forms and some properties related to the existence of $p$-K\"ahler structures on compact complex manifolds. 

Since positivity of differential forms is a pointwise notion, in the following we focus on exterior forms on vector spaces.
Let $V$ be a complex vector space of dimension $n$, denote by $\Lambda^{p,q} \coloneqq \Lambda^{p,q} V^*\subset\Lambda^{p+q}V^*_\C$ the space of $(p,q)$-forms over $V$.
In our notation, a volume form is a non-zero $(n,n)$-form, and to discuss positivity we will fix a preferred one, $\vol\in\Lambda^{n,n}\setminus\pg0$. 
We will say that a top degree form $\tau \in \Lambda^{n,n}$ is \textit{strictly positive} (respectively, \textit{positive}) if it is a positive (respectively, non negative) multiple of $\vol$.
With this notation, a volume form is positive if and only if it is strictly positive.

The notation used in the following is mostly based on \cite{hk}.
Recall that $\psi \in \Lambda^{p,0}$ is called decomposable if $\psi = \psi^{1} \wedge \dots \wedge \psi^{p}$, where $\psi^{1},\dots,\psi^p$ are $(1,0)$-forms.

\begin{definition}
    A $(p,p)$-form $\Omega$ is called \textit{strongly positive} if it can be written as 
    \begin{equation*}
        \Omega = i^{p^{2}} \sum_{j\in I} \psi_{j} \wedge \overline{\psi_{j}}=\sum_{j\in I} \pt{i\,\psi_{j}^1 \wedge \overline{\psi_{j}^1}}\wedge\dots\wedge\pt{i\,\psi_{j}^p \wedge \overline{\psi_{j}^p }}
    \end{equation*}
    where $\psi_{j} = \psi_j^{1} \wedge \dots \wedge \psi_j^{p}$ is a decomposable $(p,0)$-form, for all $j\in I$.
    The set of strongly positive $(p,p)$-forms is a convex cone, and the forms in the interior are called \textit{strictly strongly positive}. 
\end{definition}

\begin{notation}
    Throughout this paper we will adopt the notation $\tilde p \coloneqq n - p$, where $n$ denotes the complex dimension of the manifold. 
\end{notation}

Weaker conditions can be defined as follows. 

\begin{definition}
    A real $(p,p)$-form $\Omega$ is \textit{strictly positive} (respectively, \textit{positive}) if
    \begin{equation*}
        i^{\tilde p^2} \Omega \wedge \eta \wedge \overline{\eta}
    \end{equation*}
    is strictly positive (respectively, positive), for all $0 \neq \eta \in \Lambda^{\tilde p,0}$.

Furthermore, $\Omega$ is called \textit{strictly weakly positive} or \textit{transverse} (respectively, \textit{weakly positive}) if 
    \begin{equation*}
        i^{\tilde p^2} \Omega \wedge \eta \wedge \overline{\eta}
    \end{equation*}
    is strictly positive (respectively, positive), for all $0 \neq\eta \in \Lambda^{\tilde p,0}$ decomposable.
\end{definition}
For the sake of simplicity, in the following we will use the terminology of transversality, originating from \cite{sullivan}.
Note that, as for strong positivity, the sets of positive and weakly positive forms are convex cones, and the strict versions are the respective interiors.
For $p=1,n-1$ the three cones coincide. 

\begin{remark}
Transverse $(p,p)$-forms are characterized by being strictly positive forms when restricted to complex vector subspaces of complex dimension $p$ (see \cite[Criterion III.1.6]{demailly}).
\end{remark}

The discussion above allows us to introduce $p$-K\"ahler structures on complex manifolds.
Every complex manifold is orientable, and thus admits a nowhere vanishing volume form, which is enough to define positivity for differential forms, by requiring the aforementioned conditions to hold at every point.

\begin{definition}
    Let $(M,J)$ be a complex manifold of complex dimension $n$. Let $\Omega$ be a real $(p,p)$-form, where $1 \leq p \leq n$. We call $\Omega$ a $p$-K\"ahler structure if $d\Omega = 0$ and $\Omega$ is pointwise transverse. 
\end{definition}

\begin{remark}
As noted above, for $p=1,n-1$ being transverse is equivalent to being strictly strongly positive, so that $1$-K\"ahler structures coincide with K\"ahler metrics, whereas $(n-1)$-K\"ahler structures coincide with $(n-1)$-th powers of balanced metrics (see for instance \cite{michelsohn}). 
\end{remark}

The existence of $p$-K\"ahler structures can be characterized in terms of currents, \cite[Theorem 1.17]{AA}.
A particular case of this characterization is the following obstruction result.

\begin{proposition}[{\cite[Proposition 3.4]{hmt21}}]\label{OstrHMT}
    Let $(M,J)$ be a compact complex manifold of complex dimension $n$. Suppose that there exists a $(2 \tilde p-1)$-form $\alpha$ such that 
    \begin{equation*}
        0 \neq (d \alpha)^{\tilde p, \tilde p} = \sum_{j}c_{j} \psi^{j} \wedge \overline{\psi^{j}},
    \end{equation*}
    where $\psi^{j}$ is a decomposable $(\tilde p,0)$-form and $c_{j}$ have the same sign. Then, there are no $p$-K\"ahler structures on $(M,J)$.
\end{proposition}

A special instance, which will be useful later on, can be stated as follows:

\begin{remark}[{\cite[Remark 2.3]{AB}}]\label{86volte}
On a compact  $p$-K\"ahler manifold  $M$, there exist no non-zero exact decomposable holomorphic $\tilde p$-forms.
\end{remark}

\subsection{\texorpdfstring{$p$}{p}-K\"ahler structures on compact quotients of Lie groups}
In this section, we recall some results concerning the existence of $p$-K\"ahler structures on compact quotients of Lie groups. 
More precisely, we focus on compact complex manifolds $(M,J)$ such that $M = \Gamma \backslash G$, where $G$ is a simply connected Lie group of real dimension $2n$, $\Gamma \subseteq G$ is a discrete, co-compact subgroup and $J$ is a left-invariant complex structure. 

By a symmetrization argument, the existence of a $p$-K\"ahler structure on $(\Gamma \backslash G,J)$ implies the existence of a $p$-K\"ahler structure on the Lie algebra $\g$ of $G$.

\begin{lemma}[{\cite{FG,fm}}]\label{simmetrizzazioni}
    If $(\Gamma \backslash G,J)$ as above admits a $p$-K\"ahler structure, for some $1\le p\le n$, then $(\g,J)$ admits a $p$-K\"ahler structure. 
\end{lemma}
Thanks to Lemma \ref{simmetrizzazioni}, in the following sections we will work at the Lie algebra level. 
If the Lie group $G$ is solvable, or nilpotent, the compact quotient $M=\Gamma\backslash G$ is called a \textit{solvmanifold}, or \textit{nilmanifold}, respectively.
We recall that on a nilmanifold $\Gamma \backslash G$, we can consider a special class of complex structures, defined in terms of their compatibility with the ascending central series of the Lie algebra of $G$, see \cite{cfgu00}.
\begin{definition}\label{defNilp}
An invariant complex structure $J$ on a nilmanifold $\Gamma \backslash G$ is called \textit{nilpotent} if there is a basis of left-invariant $(1,0)$-forms $\pg{\varphi^1,\dots,\varphi^n}$ such that 
    \begin{equation}\label{nilpCS}
        d \varphi^{1}= 0, \quad d\varphi^{j} \in \Lambda^2\left\langle\varphi^1,\dots,\varphi^{j-1},\varphi^{\bar1},\dots,\varphi^{\overline{j-1}}\right\rangle
            , \quad j=2, \dots, n.
    \end{equation}
We will refer to such a basis as to a \textit{$J$-adapted} basis.
\end{definition}

On nilmanifolds with a nilpotent complex structure $J$, there is an obstruction to the existence of $p$-K\"ahler structures, for a specific value of $p$ depending on $J$.

\begin{proposition}[{\cite[Theorem 2.3]{SfeTar22}}]\label{STthm}
    Let $(\Gamma \backslash G,J)$ be a nilmanifold of complex dimension $n$ equipped with a nilpotent complex structure. Let $l_{0}$ the index such that 
    \begin{equation*}
        d \varphi^{j} = 0, \quad \forall j \leq l_{0}, \quad \quad d \varphi^{l} \neq 0, \quad \forall l > l_{0}.
    \end{equation*}
    Then, if $p\coloneqq n-l_{0}$, there are no $p$-K\"ahler structures on $(\Gamma \backslash G,J)$.
\end{proposition}

\begin{remark}\label{remark di unimodularita}
    We note that the analogous statement of Proposition \ref{OstrHMT} still holds on unimodular Lie algebras.
    Indeed, the key argument of the proof is to use Stokes Theorem for integration of exact forms compact manifolds, and in the case of a unimodular Lie algebra this argument can be replaced with the well known fact that $(n,n)$-forms on a unimodular Lie algebra cannot be exact, unless null.
In particular, the analogous of Proposition \ref{STthm} holds for any nilpotent Lie algebra.
\end{remark}

In a similar fashion, we generalize some well established facts about compact complex manifolds to the analogous statements for unimodular Lie algebras.

Recall that if $M$ is a complex, $p$-K\"ahler manifold, and $N$ is a complex submanifold of $M$ of complex dimension at least $p$, then $N$ is also $p$-K\"ahler \cite[Proposition 2.7]{AA}. 
The analogous statement at the level of Lie algebras is explained in the following remark.

\begin{remark}
    Any $J$-invariant subalgebra $\mathfrak h$ of a $p$-K\"ahler Lie algebra $\mathfrak g$ is still $p$-K\"ahler.
    Indeed, for any $p$-K\"ahler  structure $\Omega$  on $(\mathfrak g,J)$, the restriction $\rest{\Omega}{\mathfrak h}$ is a transverse form  on $(\mathfrak h,J)$, which is also closed because $d_\mathfrak h({\rest{\Omega}{\mathfrak h}})={\rest{\pt{d\Omega}}{\mathfrak h}}$.
\end{remark}

On the other hand, $p$-K\"ahler structures are in some sense preserved under holomorphic submersions. 
Let $f: M \to N$ be a holomorphic submersion with $M,N$ compact complex manifolds of complex dimensions $n$ and $n-q$, respectively.
As proved in \cite[Proposition 2.6]{AA}, if $M$ is $p$-K\"ahler, for $p>q,$ then $N$ is $(p-q)$-K\"ahler.

\begin{lemma}\label{Lemma veramente molto importante con cui Asia si e portata a casa la giornata}
    Let $(\g,J)$ be a Lie algebra with complex structure such that $\g$ is a semi-direct product of $J$-invariant Lie algebras $\g=\h\ltimes_\rho\kk$, and assume that $\kk$ is unimodular and $\rho\colon \h\to\operatorname{End}(\kk)$ takes values in traceless endomorphisms.
    Let $n$ and $q$ be the complex dimensions of $\g$ and $\kk$, respectively.
    If $(\g,J)$ is $p$-K\"ahler, for some $p>q,$ then $(\h,J)$ is $(p-q)$-K\"ahler.
\end{lemma}

\begin{proof}
    Let $\Omega$ be a $p$-K\"ahler structure on $(\g,J)$, and write
    \begin{equation}\label{oml}
        \Omega=\sum_{l=0}^{N}\Omega_l,\quad\Omega_l\in\Lambda^{2(p-q)+l}\h^* \wedge\Lambda^{2q-l}\kk^*,
    \end{equation}
    with $N=\min\pg{2(n-p),2q}.$
    In particular, since $\kk$ has dimension $2q$, and comparing types, we get that $\Omega_0=\Psi\wedge\vol_\kk$, for some $\Psi\in\Lambda^{p-q,p-q}\h^*_\C$ and $0\neq\vol_\kk\in\Lambda^{q,q}\kk^*_\C$.
    Furthermore, by iterations of \cite[Theorem 3.4]{FagMai}, $\Psi$ is transverse.
    The upshot is that $\Psi$ is also closed with respect to the differential operator of $\h$, which we denote with $d_\h$, thus completing the proof.
    To prove this, we decompose $d\Omega$ as in \eqref{oml}, and consider the component $(d\Omega)_0$ having degree $2q$ in $\kk$. 
From the closure of $\Omega$, and comparing types, we find
\begin{equation}\label{dom0}
    0=(d\Omega)_0=d\Psi\wedge\vol_\kk+\Psi\wedge d\vol_\kk+(d\Omega_1)_0,
\end{equation}
where once again $(d\Omega_1)_0$ denotes the component of $d\Omega_1$ in $\Lambda^{2(p-q)+1}\h^* \wedge\Lambda^{2q}\kk^*$.
Now, $d\vol_\kk=0$ is equivalent to the action $\rho$ defining the semidirect product taking values in traceless endomorphisms, so to conclude the proof it is enough to show that $(d\Omega_1)_0=0$.
Indeed, if this were to hold, \eqref{dom0} would read $0=d\Psi\wedge\vol_\kk=d_\h\Psi\wedge\vol_\kk$, implying that $d_\h\Psi=0.$
To prove  that $(d\Omega_1)_0=0$, we write $\Omega_1=\sum_j\alpha_j\wedge\beta_j$, for some $\alpha_j\in\Lambda^{2(p-q)+1}\h^*$, $\beta_j\in\Lambda^{2q-1}\kk^*$, and note that
\begin{equation*}
    (d\Omega_1)_0=\sum_j(d\alpha_j\wedge\beta_j-\alpha_j\wedge d\beta_j)_0
    =-\sum_j\alpha_j\wedge d_\kk\beta_j,
\end{equation*}
where $(d\alpha_j\wedge\beta_j)_0=0$ because $d\pt{\Lambda\h^*}\subset\Lambda\h^*$, being $\kk$ an ideal.
Recall now that $\kk$ is unimodular, so $d_\kk\beta_j$ has to vanish, for all $j$, since it is a $2q$-form in dimension $2q.$
\end{proof}

The assumptions of the Lemma are for instance satisfied by direct products $\g=\h\oplus\kk$, with $\kk$ unimodular.

\medskip

Among compact quotients of Lie groups by lattices, examples and properties of $p$-K\"ahler structures are know best on compact {\em holomorphically parallelizable} manifolds, thanks to the extensive work by Alessandrini and Bassanelli.
Recall that a complex manifold $(M,J)$ is called holomorphically parallelizable if $T^{1,0}M$ is trivial as a holomorphic bundle. 
It is known that a compact complex manifold is holomorphically parallelizable if and only if it is the compact quotient of a connected, simply-connected, complex Lie group by a lattice, see \cite{wang}. 
In such manifolds, characterization in terms of currents can be refined in terms of forms. 

\begin{proposition}[{\cite[Theorem 3.2]{AB}}]\label{AB3.2}
    Let $(M,J)$ be a compact holomorphically parallelizable manifold of complex dimension $n$ and let $1 \leq p \leq n-1$. Then, the following statements are equivalent:
    \begin{itemize}
        \item $M$ admits a $p$-K\"ahler structure;
        \item there are no non-zero, decomposable, exact and holomorphic $(\tilde p,0)$-forms.
    \end{itemize}
\end{proposition}

As per to Remark \ref{remark di unimodularita}, an analogous version of Proposition \ref{AB3.2} holds for unimodular Lie algebras endowed with a biinvariant complex structures.

\smallskip
To conclude this section, we recall an example of a family of holomorphically parallelizable nilmanifolds admitting admitting $p$-K\"ahler structures.

\begin{example}\label{exetabeta}
    In \cite[Section 4]{AB} the authors study the holomorphically parallelizable nilmanifolds $\eta\beta_{2m+1}=(\Gamma \backslash \operatorname{H}_{2m+1},J)$, compact quotients of the complex Heisenberg group of (complex) dimension $2m+1$, endowed with the standard complex structure induced by multiplication by $i$. 
    As it turns out, $\eta\beta_{2m+1}$ has $p$-K\"ahler structures if and only if $p \ge m + 1$. 
    Furthermore, it admits a global coframe of $(1,0)$-forms $\{\varphi^{1}, \dots, \varphi^{2m +1}\}$ such that 
    \begin{equation}\label{etabeta}
        \begin{cases}
            d\varphi^j=0,   &   j\le 2m ,\\
        d \varphi^{2m+1}= \sum_{s=1}^{m} \varphi^{2s-1} \wedge \varphi^{2s}.
        \end{cases}
    \end{equation}
\end{example}

\section{Positivity in higher degree}\label{secPos}

In this section we will prove preliminary results on transversality of exterior forms on complex vector spaces, and more precisely when a higher degree positive form can be constructed, starting from a lower degree one.

\smallskip
We begin studying powers of transverse forms.
Indeed, it is known that the wedge product of strongly positive forms is still strongly positive (\cite[Proposition III.1.11]{demailly}), whereas the same does not hold for weak positivity.
Indeed, positivity of the square of weakly positive forms is studied in  \cite{BlockiPlis}, and in what follows, we prove analogous results on transverse forms.

\begin{lemma}\label{a^2transv}
    For every transverse $(2,2)$-form $\alpha$ in $\C^4$,  its square $\alpha^2$ is a positive volume form.
\end{lemma}

\begin{proof}
The proof follows the ideas of the proof of \cite[Theorem 1]{BlockiPlis}, where the analogous statement is proved, for weakly positive forms.
We briefly recall the strategy of the latter.
Firstly, a preferred basis $\mathcal{B}$ of the space of $(2,0)$-forms is fixed, allowing to associate to every $(2,2)$-form $\alpha$ a Hermitian $6\times6$ matrix $A_\alpha$, where $6=\dim\Lambda^{2,0}\pt{\C^4}^*$.
Furthermore, the space of decomposable forms is isomorphic to the image of the Pl\"ucker embedding of the Grassmannian of complex $2$-vector spaces in $\C^4$ into $\mathbb P\Lambda^{2,0}\pt{\C^4}^*$.
Writing down explicitly the equations defining the Pl\"ucker quadric ${\mathcal P}$, the upshot is that $\alpha$ is weakly positive, or transverse, if $A_\alpha$ is positive semi-definite, or positive definite, respectively, on the quadric $\pg{z=\pt{z_1,\dots,z_6}\colon z_1z_6+z_2z_5+z_3z_4=0}$.
Up to a change of basis $\mathcal{B}$ with a new one, adapted to $\alpha$ (see \cite[Lemma 6]{BlockiPlis}), the problem can be reduced to a similar one, on a $4\times 4$ block of $A_\alpha$, as in \cite[Theorem 7]{BlockiPlis}.

In our setting, we can still use \cite[Lemma 6]{BlockiPlis}, to conclude that the statement of the Lemma is equivalent to show that for every Hermitian $4\times4$ matrix $A=(a_{j,k})$, if
\begin{equation}\label{pluckerino}
    \bar zAz^t>0,\quad\text{for all }z=(z_1,\dots,z_4)\neq0\text{ with }z_1z_4+z_2z_3=0,
\end{equation}
then
\begin{equation*}
    \sum_{j,k=1}^4a_{j,k}a_{5-j,5-k}>0.
\end{equation*}
In order to prove this, it is enough to repeat the computations in the proof of \cite[Theorem 7]{BlockiPlis}, choosing particular instances of $z$ in \eqref{pluckerino}.
The strict inequality in the assumption \eqref{pluckerino} guarantees the strict inequality in the thesis.
\end{proof}

\begin{corollary}\label{corSquare2}
    Let $n\ge4$, and $\Omega$ a transverse $(2,2)$-form on $\C^n$.
    Then, $\Omega^2$ is a transverse $(4,4)$-form on $\C^n$.
\end{corollary}

\begin{proof}
    We use the characterization of transversality via restriction to complex subspaces.
    More precisely, $\Omega^2$ is transverse if and only if, for every complex subspace $V$ of $\C^n$ of dimension $4$, $\rest{\Omega^2}{V}$ is a positive volume form.
    This is true for every such $V$, because $\rest{\Omega^2}{V}=\pt{\rest{\Omega}{V}}^2$, and using Lemma \ref{a^2transv} for $\alpha=\rest{\Omega}{V}$.
\end{proof}

We will now show that the same stability of transversality with respect to squares does not hold for higher degree forms.

\begin{lemma}\label{a^2NOtra}
    For every $p\ge3$, there exists a transverse $(p,p)$-form $\beta_p$ on $\C^{2p}$ such that  $\beta_p^2$ is a negative volume form.
\end{lemma}

\begin{proof}
    By \cite[Theorem 2]{BlockiPlis}, for every $p\ge3$, there exists a weakly positive $(p,p)$-form $\alpha_p$ on $\C^{2p}$ such that  $\alpha_p^2$ is a negative volume form, namely $-\alpha_p^2$ is strictly positive.
    Let $\Omega$ be a transverse $(p,p)$-form on $\C^{2p}$.
    If $\Omega^2$ is strictly negative, we have the thesis.
    Otherwise, note that for every $\varepsilon>0$, the $(p,p)$-form $\alpha_{p,\varepsilon}\coloneqq\alpha_p+\varepsilon\Omega$ is transverse, and
    \begin{equation*}
        \alpha_{p,\varepsilon}^2=\alpha_p^2+\varepsilon^2\Omega^2+2\varepsilon\alpha_p\wedge\Omega
        =\alpha_p^2+\varepsilon\pt{\varepsilon\Omega^2+2\alpha_p\wedge\Omega}.
    \end{equation*}
    By assumption, $\varepsilon\Omega^2+2\alpha_p\wedge\Omega$ is a real volume form, and since $\alpha_p^2$ is strictly negative, there is always a small enough $\varepsilon_p$ such that $\alpha_{p,\varepsilon_p}^2$ is strictly negative, so one can choose $\beta_p=\alpha_{p,\varepsilon_p}$ to satisfy the statement.
\end{proof}

\begin{corollary}\label{corSquare3}
For every $p\ge3$, $n\ge2p$, there exists a transverse $(p,p)$-form $\theta$ on $\C^{n}$ such that  $\theta^2$ is not weakly positive.
\end{corollary}

\begin{proof}
    Notice that, if $n=2p$ the thesis follows by \Cref{a^2NOtra}. Thus, in the following, we suppose that $n > 2p$. 
    
    Let us fix a basis $\pg{\varphi^1,\dots,\varphi^{n}}$ of $\Lambda^{1,0}\pt{\C^n}^*$.
    We consider the immersion of $\C^{2p}$ in the subspace $V$ of $\C^n$ such that $\Lambda^{1,0}V^*$  is spanned by $\pg{\varphi^1,\dots,\varphi^{2p}}$.
    Let $\beta_p$ be the transverse form on $\C^{2p}$ obtained at the end of the previous proof.
    Then, $\beta_p$ can be pushed forward to a weakly positive form in $\Lambda^{p,p}V^*\subset\Lambda^{p,p}\pt{\C^n}^*$, which we will still denote with $\beta_p$.

Define $\theta_\varepsilon:=\beta_p+ \varepsilon \, \eta \in\Lambda^{p,p}\pt{\C^n}^*$, where $\eta$ is a strictly strongly positive form and $\varepsilon > 0$.
In particular, $\eta$ is transverse, thus so is $\theta_\varepsilon$, for all $\varepsilon>0$, because it is the sum of a transverse form and a weakly positive one.

Now, we are left to prove that there exists $\varepsilon_0$ such that $\theta_{\varepsilon_0}^2$ is not weakly positive.
Once again, we use the characterization via restrictions to complex subspaces of $\C^n$.
More precisely, we prove that for $\varepsilon$ small enough, $\rest{\theta_\varepsilon^2}V$ is a  negative volume form on $V$, whereas there exists $ \tilde V\neq V$ of $\C^n$ of dimension $2p$ such that $\rest{\theta_\varepsilon^2}{\tilde V}$ is a positive volume form on $\tilde V$, for every $\varepsilon>0$.

Since $\theta_\varepsilon^2=\beta_p^2+ \varepsilon \pt{\varepsilon \eta^{2} + 2 \beta_{p} \wedge \eta}$, and $\rest{\beta_p^2}{V}$ is a negative volume form by \Cref{a^2NOtra}, it follows that there exists $\varepsilon_0$ small enough so that $\rest{\theta_{\varepsilon_0}^2}{V}$ is a negative volume form.

If we consider the subspace $\tilde V$ of $\C^{n}$ such that $\Lambda^{1,0} \tilde V^{\ast}$ is spanned by $\{\varphi^{2}, \dots, \varphi^{2p+1}\}$, then the restriction $\beta_p^2|_{\tilde V}$ vanishes and so for every $\varepsilon>0$
\begin{equation*}
    \rest{\theta_\varepsilon^2}{\tilde V}
    ={\rest{2\varepsilon\pt{\beta_p\wedge\eta}}{\tilde V}+ \varepsilon^2\rest{\eta^2}{\tilde V}}.
\end{equation*}
Since $\eta$ is strictly strongly positive, so is $\eta^2$.
In particular, since $\eta^2$ is also transverse, $\rest{\eta^2}{\tilde V}$ is a positive volume form on $\tilde V$.
Furthermore, $\beta_p\wedge\eta$ is weakly positive by \cite[Proposition III.1.11]{demailly}, so $\rest{\pt{\beta_p\wedge\eta}}{\tilde V}$ is a positive top form on $\tilde V$, thus allowing to conclude.
\end{proof}

We conclude this section showing one way to construct a transverse $(p+1,p+1)$-form starting from a transverse $(p,p)$-form, for every $p\le n-1$.

\begin{lemma}\label{ambrozia}
    Let $\Omega$ be a transverse $(p,p)$-form on a vector space of complex dimension $n$, and let $k\ge n-p$.
    Then, for every $\varphi^1,\dots,\varphi^k$ linearly independent $(1,0)$-forms,
    \begin{equation*}
        \widetilde{\Omega}=\Omega\wedge i\,\pt{\varphi^1\wedge\overline{\varphi^1}+\dots+\varphi^k\wedge\overline{\varphi^k}}
    \end{equation*}
    is a transverse $(p+1,p+1)$-form.
\end{lemma}

\begin{proof}
    We have to show that for all non-zero decomposable $(n-p-1,0)$-forms $ \eta$, the $(n,n)$-form 
    \begin{equation}\label{topolino}
        i^{(n-p-1)^2}\widetilde{\Omega}\wedge\eta \wedge \overline{\eta} = i^{(n-p-1)^2}\sum_{j=1}^k\Omega\wedge i\,\varphi^{j} \wedge \overline{\varphi^j} \wedge \eta \wedge \overline{\eta},
    \end{equation}
    is strictly positive.
    Since $\Omega$ is transverse, and $\varphi^j\wedge\eta$ is decomposable, for all $j=1,\dots,k$, the last summand in \eqref{topolino} is positive, and it is zero if and only if $\varphi^{j} \wedge \eta=0$ for every $j=1,\dots,k$.
    Let $\eta_1,\dots,\eta_{n-p-1}$ be $(1,0)$-forms such that $\eta=\eta_1\wedge\dots\wedge\eta_{n-p-1}$.
    For every $j=1,\dots,k$, by $\varphi^{j} \wedge \eta=0$ it follows that $\varphi^{j}$ is in the vector space spanned by $\eta_1,\dots,\eta_{n-p-1}$.
    However, this space has dimension $n-p-1<k$, in contradiction with the $\varphi^j$ being linearly independent.
    It follows that the quantity in \eqref{topolino} is positive, and never zero, unless $\eta$ is.
\end{proof}

\section{\texorpdfstring{$p$}{p}-Kähler complex manifolds}\label{secCpxMfd}

In the first part of this section, we will give a generalization of the stability of $p$-K\"ahlerianity under products of compact complex manifolds.
We will then move to the setting of compact quotients of Lie groups, and generalize a known dimensionality reduction argument for the existence of $p$-K\"ahler structures, which will be useful for proofs by induction.
The last part of the section is dedicated to the proof of the Alessandrini-Bassanelli conjecture on Lie algebras equipped with a nilpotent complex structure. 

\subsection{Products of compact \texorpdfstring{$p$}{p}-Kähler manifolds}
Given a pair of $p$-K\"ahler manifolds, the problem of their product admitting $q$-K\"ahler structures, for some $q$, is solved just assuming the Alessandrini-Bassanelli conjecture on both factors, together with further positivity assumptions.
When $p$ is high enough, we show that the further positivity assumptions are not needed for this construction.

\begin{theorem}\label{superprodotti}
    Let $M^m,N^n$ be compact complex manifolds, with $m,n\ge2$, and assume that $M$ is $(m-2)$-K\"ahler and balanced, and $N$ is $(n-2)$-K\"ahler and balanced.
    Then, $M\times N$ is $(m+n-2)$-K\"ahler and balanced.
\end{theorem}

\begin{proof}
We only need to prove that $M\times N$ is $(m+n-2)$-K\"ahler.
    Let $\Omega_{p}$ be $p$-K\"ahler structures on $M$ for $p=m-2,m-1$ and $\tilde \Omega_{q}$ be $q$-K\"ahler structures on $N$ for $q = n-2, n-1$.
    The $(m+n-2,m+n-2)$-form 
    \begin{equation}\label{Psi}
        \Psi: =  \Omega_{m-2}\wedge\vol_{N} +\Omega_{m-1}\wedge\tilde\Omega_{n-1}+\vol_M\wedge\,\tilde\Omega_{n-2}
    \end{equation}
    is closed and weakly positive, because $\Omega_{m-1}$ and $\tilde\Omega_{n-1}$ are strongly positive, for dimensionality reasons.
    To prove that $\Psi$ is transverse, assume $\Psi\wedge\eta\wedge\bar\eta=0$, for $\eta=\mu_1\wedge\mu_2$, $\mu_1,\mu_2$ $(1,0)$-forms on $M\times N$.
    Since all the three summands in \eqref{Psi} are weakly positive, their wedge product with $\eta\wedge\bar\eta$ is non-negative, so their sum vanishes if and only if they all do.
    Thus, writing $\mu_j=\alpha_j+\beta_j$, with $\alpha_j$ $(1,0)$-forms on $M$ and $\beta_j$  $(1,0)$-forms on $N$, for $j=1,2$, $\Psi\wedge\eta\wedge\bar\eta=0$ implies
    \begin{equation}\label{wedgissimi}\begin{aligned}
        0=&\Omega_{m-2}\wedge\vol_{N}\wedge\,\eta\wedge\bar\eta
        =\alpha_1\wedge\alpha_2\wedge\bar\alpha_1\wedge\bar\alpha_2\wedge\Omega_{m-2}\wedge\vol_{N}, \\
        0=&\vol_M\wedge\,\tilde\Omega_{m-2}
        =\vol_M\wedge\,\tilde\Omega_{m-2}\wedge\beta_1\wedge\beta_2\wedge\bar\beta_1\wedge\bar\beta_2,\\
        0=&\Omega_{m-1}\wedge\tilde\Omega_{n-1}\wedge\eta\wedge\bar\eta.
    \end{aligned}
    \end{equation}
From the first two lines, since $\Omega_{m-2}$ and $\tilde\Omega_{n-2}$ are transverse, we conclude $\alpha_1\wedge\alpha_2=0$ and $\beta_1\wedge\beta_2=0$.
It is easy to check that this implies $\eta=\alpha_0\wedge\beta_0$, for some $\alpha_0$ $(1,0)$-form on $M$ and $\beta_0$  $(1,0)$-form on $N$.
The last line in \eqref{wedgissimi} now reads
\begin{equation*}
0=\Omega_{m-1}\wedge\tilde\Omega_{n-1}\wedge\alpha_0\wedge\beta_0\wedge\bar\alpha_0\wedge\bar\beta_0,
\end{equation*}
which is equivalent to 
\begin{equation*}
\Omega_{m-1}\wedge\alpha_0\wedge\bar\alpha_0=0, \quad \text{or} \quad 
\tilde\Omega_{n-1}\wedge\beta_0\wedge\bar\beta_0=0.
\end{equation*}
Being both $\Omega_{m-1}$ and $\tilde\Omega_{n-1}$ transverse, it follows that $\alpha_0=0$ or $\beta_0=0$, so $\eta=0$, as wanted.
\end{proof}

\subsection{Quasi-nilpotent complex structures}
We start by recalling the definition of quasi-nilpotent complex structures. 
Note that this definition was introduced in \cite{LPhD} in the context of nilpotent Lie algebras, but the definition can be extended to any type of Lie algebras, with no further nilpotency, or solvability assumption.
For similar notions of solvability and semisimiplicity of complex structures on Lie algebras, see \cite{FuGe26}.

\begin{definition}
A complex structure $J$ on a Lie algebra $\mathfrak g$ is called \textit{quasi-nilpotent} if the center of $\mathfrak g$ has a non-trivial $J$-invariant subspace.
\end{definition}

Let $(\g,J)$ be a nilpotent Lie algebra equipped with a quasi-nilpotent complex structure and denote by $\ab $ a $2$-dimensional $J$-invariant subspace of the center $\mathfrak z$.
Since $\ab$ is an ideal of $\mathfrak g$, the quotient $\mathfrak q=\mathfrak g/\mathfrak a$ is well defined, and being $\ab$ also $J$-invariant, the complex structure $J$ descends to the quotient, defining the complex structure $J_\q$ on $\q$.
    As a vector space, $\mathfrak q$ is isomorphic to a $J$-invariant subspace of $\mathfrak g$, and under this identification, $J_\q$ is the restriction of $J$.
For this reason, in the following we will still denote $J_\q$ with $J$.

\begin{definition}[{\cite{LPhD}}]
    If $\mathfrak k$ is a nilpotent Lie algebra of real dimension $2(n-1)$ endowed with a complex structure $K$ such that $(\mathfrak k, K)$ is isomorphic to $(\mathfrak g / \ab , J)$, we call the pair $(\mathfrak g, J)$ an \textit{$\ab $-extension} of $(\mathfrak k, K)$.
\end{definition}

We recall a known obstruction to the existence of $p$-K\"ahler structures on nilpotent Lie algebras equipped with a quasi nilpotent complex structure. 

\begin{proposition}[{\cite[Proposition 3.3]{fm}}]\label{bext}
    Let $\g$ be a nilpotent Lie algebra of real dimension $2n \ge 6$ equipped with a quasi-nilpotent complex structure $J$. 
    If $(\g,J)$ admits a $p$-K\"ahler structure, for some $2\le p\le n$, then it is the $\ab $-extension of a $(p-1)$-K\"ahler nilpotent Lie algebra.
\end{proposition}

We can generalize \Cref{bext} to a larger class of Lie algebras, maintaining the same assumption on the complex structure. 

\begin{proposition}\label{enrico}
    Assume a Lie algebra $\mathfrak g$ is endowed with a quasi-nilpotent complex structure $J$ and let $\mathfrak a$ be a $J$-invariant subspace  of the center, of real dimension $2k\ge2$.
    If $(\mathfrak g,J)$ is $p$-K\"ahler, for some $p>k$, then $(\mathfrak q:=\mathfrak g/\mathfrak a,J)$ is $(p-k)$-K\"ahler.
\end{proposition}

\begin{proof}
It is enough to prove the statement for $k=1$.
Indeed, we can always write $\mathfrak a=\mathfrak a_1\times\dots\times\mathfrak a_k$, with $\mathfrak a_j$ a $J$-invariant subspace  of the center of $\g$, of real dimension $2$.
When $k\ge 2$, we can define recursively
    \begin{equation*}
        \mathfrak q_1=\mathfrak g/\mathfrak a_1,\quad
        \mathfrak q_j=\mathfrak q_{j-1}/\mathfrak a_{j},\quad j=2,\dots,k,
    \end{equation*}
    with $\mathfrak q\cong\mathfrak q_k$.

    Assume then that $k=1$, let $\Omega$ be a $p$-K\"ahler structure on $(\mathfrak g,J)$ and $\pi\colon\g\to\q$ be the projection to the quotient.
    For every $0\neq z_1\in\mathfrak a^{1,0}$, the form $\widetilde\Omega=\pi_*\pt{i\,\iota_{z_1}\iota_{{\overline z_1}}\,\Omega}$ is well defined as an element of $\Lambda^{p-1,p-1}\mathfrak q_\C^*$.
    The upshot is that $\widetilde\Omega$ is transverse and closed on $\mathfrak q$.
    The transversality follows by \cite[Theorem 3.4]{FagMai}.
    Moreover, we claim that 
    \begin{equation}\label{crackers}
        d{\widetilde\Omega}=\pi_*\pt{i\,\iota_{z_1}\iota_{{\overline z_1}}\,d\Omega},
    \end{equation}
    from which the thesis follows, as the right hand side is zero.
    For every $v_1,\dots,v_{2p-1}\in\g_\C$,
    \begin{equation*}\begin{aligned}
    \iota_{z_1}\iota_{{\overline z_1}}&\,d\Omega\pt{v_1,\dots,v_{2p-1}}
    =d\Omega\pt{\bar z_1,{z}_1,v_1,\dots,v_{2p-1}}     \\
    &=\sum_{1\le j<k\le 2p-1}(-1)^{j+k}\,\Omega\pt{[v_j,v_k],\bar z_1,{z}_1,v_1,\dots \hat{v_j},\dots,\hat{v_k},\dots,v_{2p-1}} ,
    \end{aligned}
    \end{equation*}
    where the last equality holds because $z_1,\overline z_1$ are in the center of $\mathfrak g_\C$.
    \Cref{crackers} then follows by linearity of the double contraction $\iota_{z_1}\iota_{{\overline z_1}}$, and because $\pi$ commutes with the bracket. 
    \end{proof}

In the same setting, when the quotient $\g/\ab$ is K\"ahler, we can prove that the Alessandrini-Bassanelli conjecture holds, for bidegree high enough.

\begin{proposition}\label{Lie algebra con particolari ipotesi}
 Let $\mathfrak g$ be a Lie algebra of real dimension $2n$, with a quasi-nilpotent complex structure $J$.
 Assume $\mathfrak a$ is a $J$-invariant subspace of real dimension $2k$ of the center of $\mathfrak g$, such that $(\mathfrak q=\mathfrak g/\mathfrak a,J)$ is K\"ahler.
If $(\mathfrak g,J)$ is $p$-K\"ahler, for some $p\ge k$, then it is also $(p+1)$-K\"ahler.
\end{proposition}

\begin{proof}
    Let $\Omega$ be a $p$-K\"ahler structure on $(\mathfrak g,J)$ and $\pi\colon\g\to\q$ be the projection to the quotient. Consider $\omega = \pi^{\ast} \omega_K \in \Lambda^{1,1}(\mathfrak g^{\ast}_{\C})$, where $\omega_K$ is a K\"ahler form on $(\mathfrak q,J)$. 
Then, $\Omega\wedge\omega$ is transverse by \Cref{ambrozia}, because the complex dimension of $(\mathfrak q,J)$ is $n-k\ge n-p$, and it is closed because both $\Omega$ and $\omega_K$ are.
Thus, it is a $(p+1)$-K\"ahler structure.
\end{proof}

As a direct consequence, the Alessandrini–Bassanelli conjecture holds for some of the Lie algebras considered above.
\begin{corollary}
    The Alessandrini-Bassanelli conjecture holds on Lie algebras satisfying the assumptions of \Cref{Lie algebra con particolari ipotesi}, when $k=1,2$. 
\end{corollary}

We note that if $\g$ is as in the statement of  \Cref{Lie algebra con particolari ipotesi} and $\mathfrak g/\mathfrak a$ is nilpotent, then it is abelian, so in particular the Lie algebra $\g$ is $2$-step nilpotent.
Viceversa, a $2$-step nilpotent Lie algebra satisfies the assumptions of \Cref{Lie algebra con particolari ipotesi}, for some $k\ge1$, if and only if it is not of type I, in the notation of \cite{fs2023}.

\begin{remark}
    A similar argument to the proof of \Cref{Lie algebra con particolari ipotesi} holds for every pair $(\mathfrak g,J)$ such that the Bott-Chern number $h_{BC}^{1,0}$ is at least $n-p$, namely there are $n-p$ linearly independent closed $(1,0)$-forms.
\end{remark}

\subsection{Nilmanifolds with nilpotent complex structures}
The main result of this subsection is that on a nilmanifold endowed with a nilpotent complex structure, the Alessandrini-Bassanelli conjecture holds.

\medskip
The proof will rely on the following result, consequence of \Cref{STthm}.

\begin{lemma}\label{nonsisa}
    Let $\mathfrak g$ be a nilpotent Lie algebra endowed with a nilpotent complex structure $J$, admitting $p$-K\"ahler structures, for all $p\ge p_0$.
    Then, there exists a $J$-adapted basis $\pg{\varphi^1,\dots,\varphi^n}$ such that 
    \begin{equation}\label{eqnonsisa1}
            d \varphi^{j} = 0, \quad j = 1, \dots,n-p_0+1.
    \end{equation}
\end{lemma}

\begin{proof}
By \Cref{STthm}, if $(\mathfrak g,J)$ as in the statement admits a $p$-K\"ahler structure, and  $\mathcal{B}=\pg{\varphi^1,\dots,\varphi^n}$ is a $J$-adapted basis such that 
    \begin{equation*}
            d \varphi^{j} = 0, \quad j = 1, \dots,n-p,
    \end{equation*}
    then $d\varphi^{n-p+1}=0$.
As a direct consequence we have the statement for  $p_0=n-1$, as by \eqref{nilpCS}, $d\varphi^1=0$ in every $J$-adapted basis.
For lower $p_0$, the thesis follows arguing recursively.
\end{proof}

We are now ready to prove that the Alessandrini-Bassanelli conjecture holds, on nilmanifolds endowed with a nilpotent complex structure.
By \Cref{simmetrizzazioni}, this is equivalent to the analogous statement, at the Lie algebra level.

\begin{theorem}\label{polis}
Let $\g$ be a nilpotent Lie algebra equipped with a nilpotent complex structure $J$.
If we assume that $(\g,J)$ has complex dimension $n$ and that it admits $p$-K\"ahler structures for some $p<n$, then it also admits a $(p+1)$-K\"ahler structure.
\end{theorem}

\begin{proof}
We prove the theorem by induction on the complex dimension $n$.
Note that the cases $n\le 3$ and $p=1,n-1$ are trivial, whereas $n = 4,5$ and $p=n-2$ are true by \cite[Proposition 3.5]{fm} and \cite[Theorem 4.2]{Log2025}.

Assume the thesis holds for complex dimension $n$, and let us prove it for complex dimension $n+1$.
As mentioned above, we know that the statement is true when $p=1,n-1,n$, so suppose that $\mathfrak{g}$ admits a $p$-K\"ahler structure, for $1 < p < n-1$, and recall that  $\tilde p = n + 1 - p$. 
Then, $\mathfrak g$ is the $\ab $-extension of a Lie algebra $\mathfrak k$, with $\dim_{\R}\mathfrak k=2n$, endowed with the quotient complex structure $J_\mathfrak k$, which is nilpotent.
Furthermore,  by \Cref{bext}, $(\mathfrak k,J_\mathfrak k)$ is $(p-1)$-K\"ahler. 
In particular, we can choose a $J$-adapted basis $\mathcal{B}=\pg{\varphi^1,\dots,\varphi^{n+1}}$ for $\mathfrak g$ such that $(\mathfrak k,J_\mathfrak k)$ is isomorphic to the $J$-invariant vector subspace $V_\mathfrak k$ of $\g$ with $\Lambda^{1,0}V_\mathfrak k$ generated by $\pg{\varphi^1,\dots,\varphi^n}$, and under this isomorphism, $J_\mathfrak k=\rest{J}{V_\mathfrak k}$.
By induction, $(\mathfrak k,J_{\mathfrak k})$ admits $q$-K\"ahler structures for $q= p-1, \dots, n$, so by \Cref{nonsisa}, we can choose $\mathcal{B}$ so that
     \begin{equation}\label{catusino1}
         d \varphi^{j} = 0, \quad j = 1, \dots, {n-(p-1)+1=\tilde p+1}.
     \end{equation}
Denote by $\Omega$ the $p$-K\"ahler structure on $(\mathfrak{g}, J)$, and consider the $(p+1,p+1)$-form
    \begin{equation*}
        \widetilde{\Omega} \coloneqq i\,{\Omega \wedge {\sum_{j=1}^{\tilde p} \varphi^{j} \wedge \varphi^{\overline{j}}}} .
    \end{equation*}
    By \eqref{catusino1}, $d\widetilde{\Omega} =0$, and $\widetilde{\Omega}$ is transverse, by \Cref{ambrozia}.
\end{proof}

\Cref{polis} holds in particular for holomorphically parallelizable nilmanifolds, as any invariant complex structure on such manifolds is nilpotent.

\begin{remark} 
The analogue of the Alessandrini-Bassanelli conjecture cannot hold for $p$-pluriclosed structures, namely $(p,p)$-forms that are transverse and $\partial\bar\partial$-closed.
More precisely, \Cref{polis} cannot be generalized to hold for $p$-pluriclosed structures.
Indeed, we consider the $8$-dimensional nilpotent Lie algebra $\g$ in \cite[Example 4.5]{AGL} (see also \cite[Section 4]{EFV}).
The Lie algebra $\mathfrak g$ is equipped with a nilpotent complex structure $J$, such that the complex structure equations are
    \begin{equation*}
        d\varphi^1=d\varphi^2=0,\quad
        d\varphi^3=\varphi^{1\bar1}+\frac12\varphi^{2\bar2},\quad
        d\varphi^4=-\varphi^{1\bar2},
    \end{equation*}
    where we are using the notation $\varphi^{l\bar m}:=\varphi^{l}\wedge\varphi^{\bar m} $.
    A diagonal Hermitian metric $\omega=i\pt{a_1\varphi^{1\bar1}+a_2\varphi^{2\bar2}+a_3\varphi^{3\bar3}+a_4\varphi^{4\bar4}}$ has 
    \begin{equation*}
        dd^c\omega=\pt{a_3-a_4}\varphi^{1\bar12\bar2}.
    \end{equation*}
    In particular, for $a_3=a_4$ this gives a pluriclosed metric. On the other hand, whenever $a_3\neq a_4$, the $(2,2)$-form $dd^c\omega$ is an obstruction to the existence of $2$-pluriclosed structures on $(\mathfrak g,J)$, because $\mathfrak g$ is nilpotent, thus unimodular.
Furthermore, by \cite{malcev}, the simply connected nilpotent Lie group with Lie algebra $\g$ admits lattices, because the structure constants of $\g$ are rational, so the above discussion also provides examples of nilmanifolds not satisfying the modified conjecture.

We also note that the almost abelian Lie algebras studied in \cite[Example 4.18]{MaMo26} provide further counterexamples to the analogous of the Alessandrini-Bassanelli conjecture in this setting.
Indeed, these are examples of odd complex dimension $n=2k+1$, for $k\ge2$, which are $p$-pluriclosed if and only if $p=k$.
However, the problem of existence of lattices on the associated simply connected Lie group is still open.
\end{remark}

\section{Compact holomorphically parallelizable manifolds}\label{secHolPar}

In this section, we address several problems concerning the existence of $p$-K\"ahler structures on compact holomorphically parallelizable manifolds, namely compact complex manifolds $(M,J)$ admitting a global holomorphic frame.
By \cite{wang}, such manifolds are biholomorphic to a quotient of a complex Lie group $G$ by a discrete subgroup.
Thus, the complex structure $J$ is \textit{biinvariant}, meaning that both left and right multiplication are holomorphic.
This condition, at the level of the (complex) Lie algebra $\mathfrak g$ of $G$, reads
\begin{equation*}
    [X,JY] = J[X,Y],\qquad\forall X,Y \in\g,
\end{equation*}
or equivalently, $[\g^{1,0},\g^{0,1}]=0,$ where $\g^{1,0}\subset\g_\C$ is the $i$-eigenspace of $J$, and $\g^{0,1}=\overline{\g^{1,0}}$.
In particular, $J$ induces an isomorphism of Lie algebras $\g^{1,0}\simeq\g$, which extends to an isomorphism 
\begin{equation}\label{k0isok}
    \Lambda^{k,0}\g^*_\C\simeq\Lambda^k\g^*
\end{equation}

\subsection{The solvable case}
In this section, we study the existence of $p$-K\"ahler structures on complex solvable Lie algebras with biinvariant complex structures. More precisely, we prove the Alessandrini-Bassanelli conjecture for unimodular solvable Lie algebras equipped with a biinvariant complex structure and as a consequence we characterize the existence of $p$-K\"ahler structures, for low enough $p$, in terms of the structure equations of the Lie algebra. We conclude with some remarks on Bott-Chern classes, and complex almost abelian solvmanifolds.

\smallskip
Recall that on unimodular, solvable, complex Lie algebras, with biinvariant complex structures, by \cite[(1.2)]{nakamura}, there is always a basis of $(1,0)$-forms such that 
    \begin{equation}\label{Nakamura basis}
        d \varphi^{j} = \xi_{j} \wedge \varphi^{j} + \eta_{j}, \quad j=1, \dots, n,
    \end{equation}
    where $\xi_{j} \in \langle \varphi^{1}, \dots, \varphi^{j-1} \rangle$ and $\eta_{j} \in \Lambda^{2}\langle \varphi^{1}, \dots, \varphi^{j-1} \rangle$, for all $j=1, \dots, n$. 

\smallskip
Similarly to the case of nilpotent complex structures on nilpotent Lie algebras, we show that the existence of $p$-K\"ahler structures in this setting is obstructed, for a specific value of $p$, depending on the structure equations.

\begin{proposition}
    Assume $\mathfrak g$ is a non-abelian, unimodular, solvable Lie algebra with a biinvariant complex structure $J$. 
    Let $\pg{\varphi^{1}, \dots, \varphi^{n}}$ be a basis of $(1,0)$-forms for $(\g,J)$ as in \eqref{Nakamura basis}, and let $l_0$ be the index such that 
    \begin{equation*}
        d \varphi^{l_0} \neq 0, \quad d \varphi^{j} = 0, \quad \forall j < l_0.
    \end{equation*}
    Then, there are no $p$-K\"ahler structures on $(\g,J)$, for $p=n-l_0, n-l_0+1$. 
\end{proposition}

\begin{proof}
If $\xi_{l_0} \ne 0$, we can always assume that 
\begin{equation*}
    d\varphi^{l_0}=\varphi^1\wedge\varphi^{l_0}+\eta_{l_0},
\end{equation*}
with $\eta_{l_0} \in \Lambda^{2}\langle \varphi^{1}, \dots, \varphi^{l_0-1} \rangle$.
Then,
    \begin{equation*}
        d(\varphi^{2} \wedge \dots \wedge \varphi^{l_0}) = (-1)^{l_0} \varphi^{1}\wedge \dots \wedge \varphi^{l_0},
    \end{equation*}
    because $\varphi^{2}\wedge \dots \wedge \varphi^{l_0-1}\wedge\eta_{l_0}$ is a ${l_0}$-form on a space of dimension ${l_0}-1$.
This obstructs the existence of $(n-{l_0})$-K\"ahler structures.
Furthermore, 
    \begin{equation*}
            d(\varphi^{3} \wedge \dots \wedge \varphi^{l_0}) 
            =  \pt{-1}^{l_0-1} \varphi^{3} \wedge \dots \wedge \varphi^{l_0-1}\wedge d\varphi^{l_0}
            \neq0
    \end{equation*}
is decomposable as it is in $\Lambda^{l_0-1}\langle\varphi^{1}, \dots, \varphi^{l_0}\rangle$, so there are no $(n-{l_0}+1)$-K\"ahler structures.

On the other hand, if $\xi_{l_0}=0$, then $\eta_{l_0}\neq0$.
We can assume that
    \begin{equation*}
        d (\varphi^{3} \wedge \dots \wedge \varphi^{l_0}) =\pt{-1}^{l_0} \varphi^{3} \wedge \dots \wedge \varphi^{l_0-1}\wedge \eta_{l_0}
    \end{equation*}
    is non-zero, so that 
    \begin{equation*}
            d (\varphi^{3} \wedge \dots \wedge \varphi^{l_0+1})  =  \pt{-1}^{l_0-1} \varphi^{3} \wedge \dots \wedge \varphi^{l_0-1}\wedge \pt{d\varphi^{l_0}\wedge\varphi^{l_0+1}    
             - \varphi^{l_0}\wedge d\varphi^{l_0+1}}
    \end{equation*}
    is also non-zero.
 As above, these are decomposable forms by dimensional reasons, allowing us to conclude.
\end{proof}

In order to prove the Alessandrini–Bassanelli conjecture for unimodular, solvable Lie algebras with biinvariant complex structures, we will need the following technical lemma.

\begin{lemma}\label{tanti nulli}
    Assume $\mathfrak g$ is non-abelian, unimodular and solvable, with a biinvariant complex structure, and let $n\ge4$ be the complex dimension of $(\g,J)$. Suppose that $(\mathfrak g,J)$ admits a $p$-K\"ahler structure for $1<p < n-1$. Then, there exist a basis of $(1,0)$-forms $\pg{\varphi^{1}, \dots, \varphi^{n}}$ such that 
    \begin{equation}\label{new nakamura basis}
        \begin{aligned}
            & d \varphi^{l} = 0, \quad && l=1, \dots, \tilde{p}+2,  \\
            & d \varphi^{j} = \xi_{j} \wedge \varphi^{j} + \eta_{j}, \quad && j= \tilde{p}+3, \dots, n,
        \end{aligned} 
    \end{equation}
    where $\tilde{p} \coloneqq n-p$, $\xi_{j} \in \langle \varphi^{1}, \dots, \varphi^{j-1} \rangle$ and $\eta_{j} \in \Lambda^{2}\langle \varphi^{1}, \dots, \varphi^{j-1} \rangle$.
    
    In particular, on $(\g,J)$ satisfying the assumptions, there are no $2$-K\"ahler structures. 
\end{lemma}

\begin{proof}
Let us consider a basis $\{\varphi^{1}, \dots, \varphi^{n}\}$ of $(1,0)$-forms as in \eqref{Nakamura basis}, with
\begin{equation}\label{ab}
    \xi_{j} = \sum_{s=1}^{j-1} a^{j}_{s} \varphi^{s},\quad\quad \eta_{j} = \sum_{u<v<j} b^{j}_{uv} \varphi^{uv}
\end{equation}
for $a^{j}_{s},b^{j}_{uv} \in \C$. 
For every $1 \le u < v \le \tilde{p}+1$, the $(\tilde p,0)$-form $\beta \coloneqq d \pt{\varphi^{1\cdots\hat{u}\cdots\hat{v}\cdots\tilde{p}+1}} $ is exact and decomposable for dimensional reasons, because by \eqref{new nakamura basis} it is in the space spanned by $\varphi^{1}, \dots, \varphi^{\tilde{p}+1}$.
Thus, $\beta=0$, since $(\g,J)$ admits a $p$-K\"ahler structure by assumption, and $\g$ is unimodular. 
We can compute
    \begin{equation*}
        \begin{aligned}
            0=\beta = d \pt{\varphi^{1\cdots\hat{u}\cdots\hat{v}\cdots\tilde{p}+1}} 
            = &  \sum_{t=u+1}^{v-1} (-1)^{t} \varphi^{1\cdots\hat{u}\cdots t-1}\wedge d\varphi^t \wedge \varphi^{t+1\cdots \hat{v}\cdots\tilde{p}+1} \\
            & + \sum_{t=v+1}^{\tilde{p}+1} (-1)^{t-1} \varphi^{1\cdots\hat{u}\cdots \hat{v}\cdots t-1}\wedge d\varphi^t \wedge \varphi^{t+1\cdots\tilde{p}+1} ,
        \end{aligned}
    \end{equation*}
    where the second equality holds because, for $t<u$, the wedge product $\varphi^{1\cdots t-1}\wedge d\varphi^t $ has to vanish  by \eqref{new nakamura basis}.
Using again \eqref{new nakamura basis}, together with \eqref{ab}, we obtain
    \begin{equation}\label{beta=0}
        \begin{aligned}
           0= \beta 
            = & \sum_{t=v+1}^{\hat{p}+1} (-1)^{u+v+1} b_{uv}^{t} \varphi^{1\cdots\hat{t}\cdots\tilde{p}+1} 
          +(-1)^{u+1}\pt{\sum_{t=u+1, t\neq v}^{\tilde{p}+1} a^{t}_{u}} \varphi^{1 \cdots \hat{v}\cdots \tilde{p}+1} 
            \\ & 
            +(-1)^v \pt{\sum_{t=v+1}^{\tilde{p}+1} a^{t}_{v}}  \varphi^{1 \cdots \hat{u}\cdots \tilde{p}+1} .
        \end{aligned}
    \end{equation}
    In particular, $b^{t}_{uv} = 0$ for all $t = v+1, \dots, \tilde{p}+1$, and since $u,v$ are arbitrary integers, we get that $\eta_{j}=0$ for all $j=1, \dots, \tilde{p}+1$. 
    Moreover, by \eqref{beta=0}, and once again because $u,v$ are unrestricted, it follows that
    \begin{equation}\label{system}
        \sum_{\substack{t=k+1\\t\neq q}}^{\tilde{p}+1} a^{t}_{k} = 0, \quad \sum_{t=l+1}^{\tilde{p}+1} a^{t}_{l} = 0,
    \end{equation}
    for all $k,l,q\in\pg{1, \dots, \tilde p + 1}$, with $k<q$.
Subtracting the first equation from the second one, for $k=l$, we obtain that $a_k^q=0$, for all $1 \le k<q \le \tilde{p}+1$.
Taking this into account in \eqref{ab}, we get $\xi_{j} = 0$  and thus $d \varphi^{j} = 0$, for all $j=1, \dots, \tilde p + 1$.

We are left to prove that $d \varphi^{\tilde{p}+2} =0$.
We can always assume that 
    \begin{equation*}
        \eta_{\tilde{p}+2} = \sum_{s=1}^{r_0} \varphi^{2s - 1} \wedge \varphi^{2s},
        \quad
        0\le r_0\le\left\lfloor{\frac{\tilde{p}+1}{2}}\right\rfloor,
    \end{equation*}
    where $r_0=0$ denotes the case $\eta_{\tilde{p}+2} =0.$
    By $d^{2} \varphi^{\tilde{p}+2} = 0$, we have that $\xi_{\tilde{p}+2}\wedge\eta_{\tilde{p}+2}=0$.
    If $r_0>1$, then this implies $\xi_{\tilde{p}+2} = 0$, indeed from \eqref{ab}, we obtain
    \begin{equation}\label{aggiunto ora}
        0=\xi_{\tilde{p}+2}\wedge\eta_{\tilde{p}+2}= \sum_{s=1}^{r_{0}} \xi_{\tilde{p}+2} \wedge \varphi^{2s-1} \wedge \varphi^{2s},
    \end{equation}
    Since $\xi_{\tilde{p}+2}$ is a $1$-form, all the summands in the right hand side of \eqref{aggiunto ora} are linearly independent, so they all have to vanish.
    If $r_0>1$, in particular $\xi_{\tilde{p}+2} \wedge \varphi^{1} \wedge \varphi^{2}=\xi_{\tilde{p}+2} \wedge \varphi^{3} \wedge \varphi^{4}=0$, hence $\xi_{\tilde p +2} = 0$.
    
    In this case, 
    \begin{equation*}
    d\pt{\varphi^3\wedge\dots\wedge\varphi^{\tilde{p}}\wedge\varphi^{\tilde{p}+2}}
    =\varphi^1\wedge\dots\wedge\varphi^{\tilde{p}}
    \end{equation*}
is a non vanishing $(\tilde{p},0)$-form, which contradicts the existence of $p$-K\"ahler structures on $(\g,J)$. 
Thus, $\eta_{\tilde p + 2} = a\varphi^{12}$, for $a\in\pg{0,1}$ and $d \varphi^{\tilde{p}+2} \in \Lambda^{2}\langle\varphi^{1}, \varphi^{2}, \varphi^{\tilde p + 2} \rangle$ is decomposable.
In particular,
\begin{equation*}
    d \pt{\varphi^{\tilde{p}+2} \wedge \varphi^{3} \wedge \dots \wedge \varphi^{\tilde{p}}}
 =    \pt{d\varphi^{\tilde{p}+2}} \wedge \varphi^{3} \wedge \dots \wedge \varphi^{\tilde{p}}
\end{equation*}
is also decomposable and of type $(\tilde p,0)$, so it has to vanish because $(\g,J)$ is $p$-K\"ahler and $\g$ is unimodular. 
This implies that $d \varphi^{\tilde{p}+2} =0$ because, as noted above, $d \varphi^{\tilde{p}+2}$ is in $\Lambda^{2}\langle\varphi^{1}, \varphi^{2}, \varphi^{\tilde p + 2} \rangle$.
\end{proof}

Before moving on to the proof of the Alessandrini-Bassanelli conjecture, we show that the bound $\tilde p + 2$ in \Cref{tanti nulli} is sharp, namely there is a Lie algebra satisfying the assumptions of Lemma \ref{tanti nulli}, with $ d \varphi^{j}\neq 0$, for all $j\ge \tilde{p}+3$ in \eqref{new nakamura basis}. 
The example consists in a non-nilpotent, solvable, unimodular Lie algebra of complex dimension $6$, admitting $4$-K\"ahler structures. As shown in \cite[Propositions 5.5, 5.6]{LT}, this is the lowest dimension in which examples of this kind can occur.

\begin{example}
    Let $(\g,J)$ be defined by the following complex structure equations; 
    \begin{equation*}
        \begin{cases}
             d \varphi^{j} = 0, &  j = 1, \dots, 4, \\
             d \varphi^{5} = - \varphi^{15} + \varphi^{23},  \\
             d \varphi^{6} = \varphi^{16} + \varphi^{34}.
        \end{cases}
    \end{equation*}
    We will now show that there is a $4$-K\"ahler structure on $(\g,J)$ by proving that there are no non-vanishing, decomposable and exact $(2,0)$-forms. 
    Suppose that $\sigma$ is an exact $(2,0)$-form, then 
    \begin{equation*}
        \sigma = a\pt{ - \varphi^{15} + \varphi^{23}} + b \pt{\varphi^{16} + \varphi^{34}},
    \end{equation*}
    where $a, b \in \C$.    
    Note that  
    \begin{equation*}
        \sigma^{2} = 2 \pt{a^{2} \varphi^{1235} + b^{2} \varphi^{1346} + ab \varphi^{1236} - ab \varphi^{1345}}, 
    \end{equation*}
    and so $\sigma^{2} = 0$ if and only if $a=b=0$. Since the square of a decomposable $(2,0)$-form always vanishes, the thesis follows by \Cref{AB3.2}.
\end{example}

Using \Cref{tanti nulli}, we can easily prove the Alessandrini-Bassanelli conjecture for non-abelian, unimodular solvable Lie algebras.

\begin{theorem}\label{Disney upppp}
     Assume $\mathfrak g$ non-abelian, unimodular solvable with a biinvariant complex structure, and let $n$ be the complex dimension of $(\g,J)$. Suppose that $(\mathfrak g,J)$ admits a $p$-K\"ahler structure. Then it admits a $(p+1)$-K\"ahler structure.
\end{theorem}
\begin{proof}
    Let $\{\varphi^{1}, \dots, \varphi^{n}\}$ be a basis as in \Cref{tanti nulli} and let $\Omega$ be the $p$-K\"ahler structure on $\mathfrak g$. The form defined as 
    \begin{equation*}
        \widetilde{\Omega} \coloneqq i \, \Omega \wedge \sum_{s=1}^{\tilde p} \varphi^{s} \wedge \varphi^{\overline{s}}
    \end{equation*}
    is transverse by \Cref{ambrozia}.
    Furthermore, by \Cref{tanti nulli} the $\varphi^s$ are all closed, so $\tilde\Omega$ is a $(p+1)$-K\"ahler structure on $\mathfrak g$.
\end{proof}

Next, we show that non-abelian, unimodular solvable Lie algebra with a biinvariant complex structure cannot admit $p$-K\"ahler structures, for low  $p$.

\begin{theorem}\label{n2Solv}
    On a non-abelian, unimodular solvable Lie algebra $\mathfrak g$ equipped with a biinvariant complex structure $J$, with complex dimension $n$, there are no $p$-K\"ahler structures for $p \le \lfloor{\frac{n}{2}}\rfloor$.
\end{theorem}
\begin{proof}
    By \Cref{Disney upppp} it is enough to show that there are no $\lfloor{\frac{n}{2}}\rfloor$-K\"ahler structures. Notice that, for dimensions $n=4,5$ the thesis follows from \Cref{tanti nulli}.
    
    Suppose now that $n>5$ and let $l_{0}\le n$ be the index such that $d \varphi^{l_{0}} \neq 0$ and $d \varphi^{l} = 0$, for all $l < l_{0}$.
Write $d \varphi^{l_0} = \xi \wedge \varphi^{l_0} + \eta $ as in \eqref{Nakamura basis}, with $\xi \in\langle\varphi^1,\dots,\varphi^{l_0-1}\rangle$. 
We can argue as in the proof of Lemma \ref{tanti nulli} and write
    \begin{equation}\label{etas0}
        \eta = \sum_{s=1}^{r_0} \varphi^{2s - 1} \wedge \varphi^{2s},
        \quad
        0\le r_0\le\left\lfloor{\frac{l_0-1}{2}}\right\rfloor,
    \end{equation}
    where $r_0=0$ denotes the case $\eta=0.$
    As above, by $d^{2} \varphi^{l_0} = 0$, we have $\xi = a_{1} \varphi^{1} + a_{2} \varphi^{2}$, $a_1,a_2\in\C$. 
    
    If $r_0>1$, then $\xi = 0$ and
    \begin{equation*}
        d\pt{\varphi^{l_0} \wedge \bigwedge_{k=2}^{\left\lfloor{\frac{l_0+1}{2}}\right\rfloor} \varphi^{2k}}
        =\varphi^1\wedge\bigwedge_{k=1}^{\left\lfloor{\frac{l_0+1}{2}}\right\rfloor} \varphi^{2k}
    \end{equation*}
    contradicts the existence of $p$-K\"ahler structures, for 
    $$
    p=n-\pt{\left\lfloor{\frac{l_0-1}{2}}\right\rfloor + 1}
    =n-\left\lfloor{\frac{l_0+1}{2}}\right\rfloor
    \ge n-\left\lceil{\frac{n}{2}}\right\rceil
    =\left\lfloor{\frac{n}{2}}\right\rfloor,
    $$
    where the inequality holds because $l_0\le n$.
    By Theorem \ref{Disney upppp}, this always contradicts the existence of $\lfloor{\frac{n}{2}}\rfloor$-K\"ahler structures on $(\g,J).$
    
    If $r_0\le1$ in \eqref{etas0}, then there is $a_3\in\pg{0,1}$ such that $d \varphi^{l_0} = (a_{1} \varphi^{1} + a_{2} \varphi^{2}) \wedge \varphi^{l_0} + a_3\varphi^{12}$, which is decomposable for dimensional reasons. 
    Since this is also non-vanishing by assumption, it obstructs the existence of $(n-2)$-K\"ahler structures.
    Recall that $n>5$, so $n-2\ge\lfloor{\frac{n}{2}}\rfloor$ and we conclude once again by \Cref{Disney upppp}.
\end{proof}

\begin{remark}\label{rmkn2}
    The bound $p\le{n/2}$ in \Cref{n2Solv} is sharp, namely there are examples of non-abelian nilpotent Lie algebras with biinvariant complex structures, which are $p$-K\"ahler with $p=\lfloor{n/2}\rfloor+1$, in each complex dimension $n\ge 3$.
    Indeed, for odd complex dimension, it is enough to consider the manifolds $\eta\beta_{2m+1}$, for $m\ge 1$, introduced in \Cref{exetabeta}.
    As noted above, they are $p$-K\"ahler if and only if $p\ge m+1$.
Counterexamples in even complex dimension can be obtained as follows: let $T$ be a $1$-dimensional complex torus, and observe that the direct product $\eta\beta_{2m+2}\coloneqq\eta\beta_{2m+1}\times T$ is in fact $(m+2)$-K\"ahler, by \cite[Theorem 4.5]{AB}.
\end{remark}

In the following, we characterize non-abelian, unimodular solvable Lie algebra with non-trivial center of complex dimension $n$, with biinvariant complex structures that admit $(\lfloor{n/2}\rfloor +1)$-K\"ahler structures. 
We recall that the rank of a non-zero $(2,0)$-form $\eta$ is the integer $1 \leq r_0 \le n/2$ such that 
\begin{equation*}
    \eta = \sum_{s=1}^{r_0} f^{2s-1} \wedge f^{2s},
\end{equation*}
where $\{f^{1}, \dots, f^{2r_0}\}$ are linearly independent $(1,0)$-forms.
Before showing the characterization, we need the following technical lemma.

\begin{lemma}\label{quanto l'avevi scritto male}
     Assume that $\mathfrak g$ is nilpotent and $J$ a biinvariant complex structure, such that $(\g,J)$ is $p$-K\"ahler.
     Let $\pg{\varphi^{1}, \dots, \varphi^{n}}$ be a $J$-adapted basis, with  $l_0$ the index such that 
    \begin{equation*}
        d \varphi^{j}=0, \quad j<l_{0}, \quad\quad d\varphi^{l} \ne 0, \quad  l \ge l_{0}.
    \end{equation*}
    Then, $d\varphi^{l_0}$ has rank $r_0>n-p-1$.
\end{lemma}
\begin{proof}
    We first note that, by \Cref{nonsisa} and \Cref{polis}, there exist a $J$-adapted basis such that $d \varphi^{j}=0$, for all $j = 1, \dots, n-p+1$.
    By definition of $r_0$, we can write
    \begin{equation*}
        d\varphi^{l_0}=\sum_{s=1}^{r_{0}} \eta^1_{s} \wedge \eta_{s}^{2},
    \end{equation*}
    for some linearly independent $\eta_s^k\in\langle \varphi^{1} ,\dots, \varphi^{l_0-1} \rangle$, $s=1,\dots,r_0$, $k=1,2$. Since $d\varphi^{l_0}\neq0$, $r_0>0$.
    Then,
    \begin{equation*}
        d \pt{\varphi^{l_{0}} \wedge \bigwedge_{s=2}^{r_0} \eta_s^1} = d \varphi^{l_{0}} \wedge \bigwedge_{s=2}^{r_0} \eta_s^1 = \eta_1^1\wedge \eta_1^2 \wedge \bigwedge_{s=2}^{r_0} \eta_s^1,
    \end{equation*}
     is a non-zero, decomposable $(r_0 + 1,0)$-form, so by Remark \ref{86volte} there are no $(n-r_0-1)$-K\"ahler structures on $(\g,J)$.
    By assumption, this can only happen if $n-r_0-1<p$, as wanted.
\end{proof}

We can now prove the following theorem.  

\begin{proposition}\label{lowestp}
    Let $\g$ be a non-abelian, {unimodular, solvable} Lie algebra {with non-trivial center} of real dimension $2n\ge10$, equipped with a biinvariant complex structure $J$.
    Assume $(\g,J)$ is $(m+1)$-K\"ahler, with $m=\lfloor{n/2}\rfloor$. Then,
    \begin{enumerate}[(i)]
        \item\label{odd} if $n=2m+1$ is odd, then $\mathfrak g \cong Lie(\eta \beta_{2m+1})$;
        \item\label{even} if $n=2m$ is even, then $\mathfrak g$ admits a $J$-adapted basis $\pg{\varphi^{1}, \dots, \varphi^{2m}}$ such that 
    \begin{equation*}
        \begin{cases}
            d \varphi^{j} = 0,     &  j =1, \dots, 2m-2, \\
            \displaystyle d \varphi^{2m-1} = \sum_{j=1}^{m-1} \varphi^{2j-1}\wedge \varphi^{2j}, \\
            \displaystyle d \varphi^{2m} = \sum_{j=1}^{m-1} \eta_{2j-1} \wedge \eta_{2j},
        \end{cases}
    \end{equation*}
    where the $\eta_{j} \in \langle \varphi^{1} ,\dots, \varphi^{2m-2} \rangle$ are either all $0$, or linearly independent.
    \end{enumerate}
\end{proposition}

\begin{proof}
{Since} $\mathfrak g$ has non trivial center {and $J$ is biinvariant, there is a}  $J$-invariant central ideal $\mathfrak a$ of $\mathfrak g$, of real dimension $2$. Then by {\Cref{enrico}}, $\mathfrak h = \mathfrak g / \mathfrak a$ is $m$-K\"ahler. 

If we are in case \ref{odd}, $\mathfrak h$ is {solvable} and of complex dimension $2m$, so by {\Cref{n2Solv}} it is abelian.
    Thus, there exists a $J$-adapted basis for $\g$ $\pg{\varphi^{1}, \dots, \varphi^{2m+1}}$ with $d\varphi^{j}=0$, for $j \le2m$.
    By unimodularity and by \eqref{Nakamura basis}, we can assume that $d\varphi^{2m+1}$ is spanned by $\pg{\varphi^{1}, \dots, \varphi^{2m}}$, so $\g$ is nilpotent.
    Furthermore, 
    by \Cref{quanto l'avevi scritto male}, $d \varphi^{2m+1}$ has rank $r_0$, with 
    \begin{equation*}
        m\ge r_0>(2m+1)-(m+1)-1=m-1.
    \end{equation*}
    We obtain the statement comparing this result with \eqref{etabeta}.

On the other hand, if $\g$ is in case \ref{even}, then $\h$ still satisfies the assumptions of the theorem, and is in case \ref{odd}, so by the first part of the proof $\mathfrak h \cong Lie(\eta \beta_{2m-1})$.
{Once again, $\g$ is nilpotent by unimodularity, so} there is a $J$-adapted basis $\pg{\varphi^{1}, \dots, \varphi^{2m}}$ of $\g$ such that 
    \begin{equation}\label{strEqBoh}
        \begin{cases}
            d \varphi^{j} = 0, & j =1, \dots, 2m-2, \\
            \displaystyle d \varphi^{2m-1} = \sum_{j=1}^{m-1} \varphi^{2j-1}\wedge \varphi^{2j}, &\\
            \displaystyle d \varphi^{2m} = \sum_{s=1}^{s_0} \eta_{2s-1} \wedge \eta_{2s} + \sum_{l=1}^{2m-2} a_{l} \, \varphi^{2m-1} \wedge\varphi^{l},&
        \end{cases}
    \end{equation}
    where $a_1,\dots,a_{2m-2}\in\C$ and either $s_0=0$, or $\pg{\eta_{s},1\le s\le 2s_0} \subset \langle \varphi^{1} ,\dots, \varphi^{2m-2} \rangle$ is linearly independent.
    We will show that $a_l=0$, for all $l\le 2m-2$, by the Jacobi identity. 
    Denote by $\pg{Z_{1}, \dots, Z_{2m}}$ the dual basis of $\pg{\varphi^{1}, \dots, \varphi^{2m}}$.
    For $j < m$, $l \le 2m-2$ and $l \neq 2j, 2j-1$, we have by \eqref{strEqBoh} that
    \begin{equation*}
        0=\operatorname{Jac}(Z_{2j-1},Z_{2j}, Z_{l}) =- [Z_{2m-1}, Z_{l}] = a_{l} Z_{2m}.
    \end{equation*}    
Arguing as in the proof of \Cref{quanto l'avevi scritto male}, we find that
    \begin{equation*}
        d \pt{\varphi^{2m} \wedge \bigwedge_{s=2}^{s_0} \eta_{2s-1}} = \eta_{1}\wedge\eta_{2} \wedge \bigwedge_{s=2}^{s_0} \eta_{2s-1}, 
    \end{equation*}
    which is a decomposable $(s_0+1,0)$-form. 
    If $s_0>0$, this obstructs the existence of $(2m-s_0-1)$-K\"ahler structures, allowing us to conclude that $s_0=m-1$.
    \end{proof}

Next, we show that in this setting the Bott-Chern class of a $p$-K\"ahler structure is always non-vanishing. 

\begin{proposition}\label{coooomoooology}
    Let $\g$ be a non-abelian, unimodular solvable Lie algebra with a biinvariant complex structure $J$, and let $n$ be the complex dimension of $(\g,J)$. If $(\g,J)$ admits a $p$-K\"ahler structure $\Omega$, then $[\Omega]_{BC} \neq 0$.
\end{proposition}
\begin{proof}
    If $(\g,J)$ admits a $p$-K\"ahler structure, then by \Cref{tanti nulli} there exists a basis of $(1,0)$-forms $\{\varphi^{1}, \dots, \varphi^{n}\}$ satisfying \eqref{new nakamura basis}.
    Thus the $(\tilde p,0)$-form $\eta = \varphi^{1} \wedge \dots \wedge \varphi^{\tilde p}$ is non-vanishing, decomposable and closed. 
    The thesis follows by \cite[Theorem 6.1]{Log2025}.
\end{proof}

\begin{remark}
    We recall that, on holomorphically parallelizable solvmanifolds, the Bott–Chern cohomology cannot be computed merely by means of left-invariant forms; in general, one needs the techniques developed in \cite{AngellaKasuya}.
    Nevertheless, if a $p$-K\"ahler structure exists on such a manifold, we can always show that it gives a non-vanishing class in the Bott-Chern cohomology, even when the $p$-K\"ahler structure is not left-invariant.
    Indeed, applying the symmetrization procedure (\Cref{simmetrizzazioni}), we can construct a $p$-K\"ahler structure on the Lie algebra $(\g, J)$. 
    Therefore, by \Cref{tanti nulli}, there exists a basis of $(1,0)$-forms $\{\varphi^{1}, \dots, \varphi^{n}\}$ on $(\g,J)$ satisfying \eqref{new nakamura basis}. 
    Extending these forms to left-invariant forms on the whole solvmanifold $(\Gamma \backslash G, J)$, we can apply the same argument as in \Cref{coooomoooology} to construct a non-vanishing, decomposable and closed $(\tilde p,0)$-form on $(\Gamma \backslash G, J)$.
\end{remark}

We conclude the section with the case of complex almost abelian Lie algebras.
In the same notation as \cite[Section 4]{Stanfield2021}, a complex almost abelian Lie algebra is a complex lie algebra admitting a complex codimension $1$ abelian ideal.

\begin{proposition}
    There are no $p$-K\"ahler structures for $p<n-1$ on non-abelian, unimodular complex almost abelian Lie algebras, endowed with a biinvariant complex structure, of complex dimension $n$. 
\end{proposition}
\begin{proof}
    Let $\mathfrak g$ be non-abelian, unimodular and complex almost abelian, with $J$ biinvariant. 
    Then, there exists a basis $\{\varphi^{1}, \dots, \varphi^{n}\}$ of $(1,0)$-forms such that 
    \begin{equation*}
        d \varphi^{j} = \varphi^{1} \wedge \sigma_{j}, \quad \forall j=1, \dots, n,
    \end{equation*}
    where $\sigma_{j} \in\langle \varphi^{1}, \dots , \varphi^{j-1}\rangle$. By Remark \ref{86volte}, there are no $(n-2)$-K\"ahler structures, unless the Lie algebra is abelian. By \Cref{Disney upppp} we conclude that there are no $p$-K\"ahler structures for $p<n-1$. 
\end{proof}

\subsection{The reductive case}

We will now focus on holomorphically parallelizable manifolds which are quotients of reductive complex Lie groups.
Equivalently, we will consider complex reductive Lie algebras, endowed with biinvariant complex structures.

Recall that a Lie algebra $\g$ is called reductive if it is the direct sum of an abelian Lie algebra $\ab$ and a semisimple one $\g_s$.
The results presented in this section will use the root space decomposition of semisimple Lie algebras, see for instance \cite[Sections 3--5]{helgason}.
Let $\h$ be a Cartan subalgebra of $\g_s$, $\Delta$ a set of roots for $\h$ and $B$ the Killing form of $\g_s$.
For $\alpha\in\Delta$ we will denote by $H_\alpha\in\h$ the $B$-dual of $\alpha$, and by $\g_\alpha$ the $1$-dimensional $\alpha$-eigenspace of $\operatorname{ad}_\h$.
Recall that 
\begin{equation*}
    \g_s=\h\oplus\bigoplus_{\alpha\in\Delta}\g_\alpha.
\end{equation*}
In the following, it will be useful to fix a preferred basis of $\g_s$, described as follows.

\begin{proposition}[{\cite[Theorem III.5.5]{helgason}}]\label{Xalpha}
Let $\h$ and $\Delta$ be fixed. There exist $X_a\in\g_\alpha$, $\alpha\in\Delta$, such that, for $\alpha,\beta\in\Delta$,
\begin{enumerate}[label=\roman*.]
\item $\pq{X_\alpha,X_{-\alpha}}=H_\alpha$,
\item $[ H , X_\alpha] = \alpha ( H ) X_\alpha$, for all $ H \in\h$,
\item $[X_\alpha, X_\beta] = N_{\alpha,\beta}X_{\alpha+\beta}$, if $\alpha+\beta\in\Delta$,
\end{enumerate}
with $N_{\alpha,\beta}\in\C$ such that $N_{\alpha,\beta}=-N_{-\alpha,-\beta}$. 
\end{proposition}

The constant $N_{\alpha,\beta}^2$ depends on the \textit{$\alpha$-string containing $\beta$}, namely the set 
\begin{equation*}
\pg{\alpha+k\beta\in\Delta,\,k\in\Z }.
\end{equation*}
More precisely, let $k^-,k^+$ be the unique integers such that $\alpha+k\beta\in\Delta$ if and only if $k^-\le k\le k^+$.
Note that $k^-\le 0\le k^+$, because $\alpha\in\Delta$.
Then, 
\begin{equation}
    N_{\alpha,\beta}^2=\frac{k^+(1-k^-)}{2}\alpha(H_\alpha).
\end{equation}
Furthermore,
\begin{equation}\label{k01}
    -2\beta(H_\alpha)=\pt{k^-+k^+}\alpha(H_\alpha).
\end{equation}

\begin{lemma}\label{lRoots}
If $\mathfrak g$ is a complex simple Lie algebra of rank at least $2$, then for any Cartan subalgebra $\mathfrak h$, and root system $\Delta$, every $\alpha\in\Delta$ satisfies one of the following:
\begin{enumerate}[label=(\roman*)]
    \item\label{case1} there exists $\beta\in\Delta$ with $\alpha+\beta\in\Delta$;
    \item\label{case2} there exist $\phi,\psi\in\Delta$ such that $\alpha=\phi+\psi$.
\end{enumerate}
\end{lemma}

\begin{proof}
Since $\rk(\g)\ge2$, for every $\alpha\in\Delta$ there is $\beta\in\Delta\setminus\pg{\pm\alpha}$ with $\beta(H_\alpha)\neq0$.
Then, the $\alpha$-string containing $\beta$ cannot be empty, for otherwise $k^-=k^+=0$, contradicting \eqref{k01}.
Now, if $k^+>0$, we are in case \ref{case1}, otherwise we must have $k^-<0$, namely $\alpha-\beta\in\Delta$, and $\alpha=(\alpha-\beta)+\beta$, so \ref{case2} holds.
\end{proof}

\begin{lemma}\label{exactdec}
    Let $\g$ be a complex semisimple Lie algebra such that every simple factor has rank greater than $1$, and  $\Psi$ be an exact $2$-form on $\g$.
    If $\Psi^2=0$, then $\Psi=0$.
\end{lemma}

\begin{proof}
    Let $\eta$ be a $1$-form such that $\Psi=d\eta$.
    Then, for every $v_1,\dots,v_4\in\g$, we have
    \begin{equation}\label{psi^2}
    \begin{aligned}
        \frac{\Psi^2}{2}\pt{v_1,\dots,v_4}=&\,
        \eta([v_1,v_2])\eta([v_3,v_4])-\eta([v_1,v_3])\eta([v_2,v_4])
        \\ &+ \eta([v_1,v_4])\eta([v_2,v_3]).
    \end{aligned}
    \end{equation}

    Let $\h$ be a Cartan subalgebra of $\g$, $\Delta$ a root system, and fix $H_\alpha,X_\alpha$ as in \Cref{Xalpha}.
    We now assume that $\Psi^2=0$, and we will show that $\eta(H_\alpha)=\eta(X_\alpha)=0$, for every $\alpha\in\Delta$, so in particular $\eta=0$ and $\Psi=0$.
For every $\alpha\in\Delta$, $\alpha$ is contained in one of the simple factors of $\g$, so we can use \Cref{lRoots}.

    If $\alpha$ falls in case \ref{case1}, there is some $\beta\in\Delta\setminus\pg{\pm\alpha}$ such that $\alpha+\beta\in\Delta$, then we consider the $\alpha$-string containing $\beta$,
    \begin{equation*}
        \gamma_k\coloneqq\alpha+k\beta\in\Delta,\,k^-\le k\le k^+ 
    \end{equation*}
    with $k^+\ge1$ since $\alpha+\beta\in\Delta$.
    
We prove by induction that  $\eta(X_{\gamma_k})=0$, for $k^-\le k<k^+$.
For every such $k$, let $H_0\in\ker\beta\setminus\ker\alpha$ such that $\alpha(H_0)=1$.
We note that such a $H_0$ always exists, because $\ker\alpha$ is the $B$-orthogonal complement of $H_\alpha$ in $\h$, for all $\alpha\in\Delta$, and so $\alpha\neq\beta$ if and only if $\ker \alpha\neq\ker\beta$. 
Furthermore, $\ker \alpha$ and $\ker\beta$ have same dimension, so one cannot be strictly contained in the other.

By construction, $\gamma_k(H_0)=1$, for all $k^-\le k\le k^+ .$
We will use equation  \eqref{psi^2}, with 
\begin{equation*}
    v_1=H_0,\quad v_2=X_{\gamma_k},\quad v_3=X_{-\beta},\quad v_4=X_{\gamma_{k+1}}=X_{\gamma_k+\beta}.
\end{equation*}
Then,
\begin{equation}\label{v13v23}
\begin{aligned}
    &\pq{v_1,v_3}=-\beta(H_0)X_{-\beta}=0,
    \\
    &\pq{v_2,v_3}=\begin{cases}
        0 ,  &   \text{if }k=k^-,\\
        N_{\gamma_k,-\beta}X_{\gamma_k-\beta}=N_{\gamma_k,-\beta}X_{\gamma_{k-1}},  &   \text{if }k^-<k<k^+.
    \end{cases}
\end{aligned}
\end{equation}
For $k=k^-$, if $\Psi^2=0$, by \eqref{psi^2} and \eqref{v13v23} we get
\begin{equation*}
    0=\eta([v_1,v_2])\eta([v_3,v_4])=\eta\pt{\gamma_{k^-}(H_0)X_{\gamma_{k^-}}}\eta\pt{c_{\gamma_{k^-},\beta}X_{\gamma_{k^-}}}=c_{\gamma_{k^-},\beta} \pt{\eta\pt{X_{\gamma_{k^-}}}}^2,
\end{equation*}
for some non-zero $c_{\gamma_{k^-},\beta}\in\C$.
It follows that $\eta\pt{X_{\gamma_{k^-}}}=0$.

For the inductive step, assume  that  $\eta(X_{\gamma_{k-1}})=0$, for a fixed $k^-< k<k^+$.
Then, by \eqref{v13v23}, $\eta\pt{\pq{v_2,v_3}}=0$, so  \eqref{psi^2} reads
\begin{equation*}
    0=\eta([v_1,v_2])\eta([v_3,v_4])
    =c_{\gamma_k,\beta} \pt{\eta\pt{X_{\gamma_k}}}^2,
\end{equation*}
with $0\neq c_{\gamma_k,\beta}\in\C$.
Thus $\eta\pt{X_{\gamma_k}}=0$, completing the proof of the induction argument.
Since $k^-\le 0<k^+$, it follows that $\eta\pt{X_\alpha}=0$.

On the other hand, if $\alpha$ is in case \ref{case2}, namely $\alpha=\phi+\psi,$ $\phi,\psi\in\Delta$, we can consider
\begin{equation*}
    v_1=H,\quad v_2=X_{\phi},\quad v_3=X_{\psi},\quad v_4=X_{\phi+\psi}=X_\alpha,
\end{equation*}
$H\in\h$, with 
\begin{equation*}
    \eta\pt{\pq{v_1,v_2}}=-\phi(H)\eta\pt{X_\phi}=0,
\end{equation*}
where the last equality follows from the first part of the proof, as $\phi,\psi$ are in case \ref{case1} of \Cref{lRoots}.
Similarly, $\eta\pt{\pq{v_1,v_3}}=0$, so that 
\begin{equation*}
\begin{aligned}
\Psi^2\pt{v_1,\dots,v_4}=2\eta([v_1,v_4])\eta([v_2,v_3])=2\pt{\phi+\psi}(H)N_{\phi,\psi}\pt{\eta\pt{X_{\phi+\psi}}}^2.
\end{aligned}
\end{equation*}
We can choose $H$ such that $\pt{\phi+\psi}(H)N_{\phi,\psi}\neq0$, to conclude that $\eta\pt{X_{\phi+\psi}}=0.$

To conclude the proof, we show that $\eta(H_\alpha)=0$, for every $\alpha\in\Delta$. First, suppose $\alpha $ falls under case \ref{case1}. Then there exists a $\beta\in\Delta\setminus\pg{\pm\alpha}$ such that $\alpha + \beta \in \Delta$ and we can consider
\begin{equation*}
    v_1= X_{\alpha},\quad v_2=X_{-\alpha},\quad v_3=X_{\alpha +\beta} - X_{\beta},\quad v_4=X_{-\alpha - \beta} + X_{-\beta}.
\end{equation*}
Since $\eta(X_{\gamma}) = 0$, for all $\gamma \in \Delta$, a direct computation gives
\begin{equation*}
    \begin{aligned}
        & \eta([v_{1},v_{3}]) = \eta(N_{\alpha,\alpha +\beta}X_{2\alpha+\beta} - N_{\alpha,\beta}X_{\alpha+\beta})=0, \\
        & \eta([v_{1},v_{4}]) = \eta (N_{\alpha,-\alpha -\beta}X_{\beta} + N_{\alpha,-\beta}X_{\alpha-\beta}) = 0,
    \end{aligned}
\end{equation*}
and 
\begin{equation*}
    \eta([v_{3},v_{4}])= \eta ([X_{\alpha + \beta},X_{-\alpha-\beta}] - [X_{\beta}, X_{-\beta}])= \eta (H_{\alpha + \beta} - H_{\beta}) = \eta(H_{\alpha}).
\end{equation*}
Hence,  
\begin{equation*}
\Psi^2\pt{v_1,\dots,v_4}=2\eta([v_1,v_2])\eta([v_3,v_4])= 2 \eta(H_{\alpha})^{2},
\end{equation*}
and so $\eta(H_{\alpha}) = 0$.
If instead $\alpha$ belongs to case \ref{case2}, then $\alpha = \phi + \psi$, for some $\phi, \psi \in \Delta$. Consequently, 
\begin{equation*}
    \eta(H_{\alpha}) = \eta(H_{\phi} + H_{\psi}) = \eta(H_{\phi}) + \eta(H_{\psi}) = 0,
\end{equation*}
as wanted.
\end{proof}

\begin{theorem}\label{P-K su reductive}
    Let $\g$ be a complex, non-abelian, reductive Lie algebra of complex dimension $n$, and $J$ a biinvariant complex structure on $\g$.
    Then, $(\g,J)$ admits a $(n-2)$-K\"ahler structure if and only if $\g$ has no simple factor of rank $1$.
\end{theorem}

\begin{proof}
Let $\g=\ab\oplus\g_s$ be the splitting of $\g$ as a direct product of its center $\ab$ and its semisimple part $\g_s$.
    Then, $J$ preserves both $\ab$ and $\g_s$, and we can denote with $n$ the complex dimension of $(\g,J)$, and with $n_s$ the complex dimension of $(\g_s,J)$.
    
We note that $\g$ is $(n-2)$-K\"ahler if and only if $\g_s$ is $(n_s-2)$-K\"ahler, allowing us to restrict our proof to the semisimple case.
Indeed, if $\g$ is $(n-2)$-K\"ahler, we obtain a $(n_s-2)$-K\"ahler structure on $\g_s$, by \Cref{enrico}.
Conversely, recall that $(\ab,J)$ admits K\"ahler metrics because it is abelian, and $(\g_s,J)$ admits balanced metrics because $J|_{\g_s}$ is biinvariant, namely it is $(n_s-1)$-K\"ahler.
Thus, if $(\g_s,J)$ is $(n_s-2)$-K\"ahler, we can argue as in the proof of \cite[Theorem 4.5]{AB}, to conclude that $\g$ is $(n-2)$-K\"ahler.

As per the discussion above, we can now restrict to the case where $\g$ is semisimple. If $\g$ has a simple factor of rank $1$, then this factor is isomorphic to $\sll_2\C$, and has dimension $3$. 
For the sake of contradiction, assume $\g$ has a $(n-2)$-K\"ahler structure $\Omega$. 
Note that the complement of $\sll_2\C$ in $\g$ is semisimple, thus unimodular, so we can apply Lemma \ref{Lemma veramente molto importante con cui Asia si e portata a casa la giornata}, to conclude that the factor $\sll_2\C$ admits a $(3-2)$-K\"ahler structure, namely a K\"ahler metric.
However, semisimple Lie algebras are never K\"ahler, giving a contradiction.

Conversely, assume that $\g$ has no simple factor of rank $1$, and let us prove that $(\g,J)$ admits a  $(n-2)$-K\"ahler structure.
By \Cref{AB3.2}, together with \eqref{k0isok}, this is equivalent to show that every exact, decomposable $2$-form on $\g$ is null.
Let $\Psi$ be such a $2$-form.
Since $\Psi$ is decomposable, then $\Psi^2=0$, so we are in the assumptions of \Cref{exactdec}, and we conclude that $\Psi=0$, proving the theorem.
\end{proof}

\begin{proposition}
    Let $\g$ be a complex, non-abelian, semisimple Lie algebra of complex dimension $n$, and $J$ a biinvariant complex structure on $\g$. Suppose that every simple factor has rank grater than 1.
    Then, $(\g,J)$ admits a $(n-2)$-K\"ahler structure $\Omega$ such that $[\Omega]_{BC} = 0$.
\end{proposition}

\begin{proof}
    To prove the proposition, we partially follow the proof of \Cref{AB3.2}. 
    Fix a basis of $1$-forms $\{\varphi^{1}, \dots, \varphi^{n}\}$ and consider $\Theta \coloneqq \varphi^{1} \wedge \dots \wedge \varphi^{n}$. We can define a non-degenerate bilinear form $F: \Lambda^{2}(\g^{\ast}) \times \Lambda^{n-2}(\g^{\ast})$ by 
    \begin{equation*}
        F(\eta,\psi) \Theta = \eta \wedge \psi, \quad \forall \eta \in \Lambda^{2}(\g^{\ast}), \forall \psi \in \Lambda^{n-2}(\g^{\ast}).
    \end{equation*}
    Let $W \coloneqq \langle d\varphi^{1}, \dots, d \varphi^{n} \rangle$ be the set of exact $2$-forms, and consider its orthogonal $W^{\perp}$ with respect to $F$. Note that $W^{\perp} \subseteq \Lambda^{n-2} (\g^{\ast})$ is the set of closed $(n-2)$-forms.

    As shown in the proof of \Cref{P-K su reductive}, combining \eqref{k0isok} with the fact that exact and decomposable $2$-forms on $\g$ must vanish (see \Cref{exactdec}), we obtain that $(\g,J)$ admits a $(n-2)$-K\"ahler structure. 
    Via the isomorphism \eqref{k0isok}, $W^{\perp}$ is mapped into the subspace $\tilde W^{\perp}$ of $d$-closed forms contained in $\Lambda^{n-2,0}(\mathfrak g^{\ast}_{\C})$. 
    Fix a basis $\{\psi^{1}, \dots, \psi^{l}\}$ of $\tilde W^{\perp}$, and define $\Omega \coloneqq i^{(n-2)^2}\sum_{j=1}^{l} \psi^{j} \wedge \overline{\psi^{j}}$. By construction $\Omega$ is real, closed and transverse. Moreover, since $H^{2}_{dR}(\g^{\ast}) = 0$ (see \cite[Theorem 21.1]{ChevEilen}) and $H^{n-2}_{dR}(\g^{\ast}) = 0$ by duality, it follows that every form in $W^{\perp}$ is exact. 
    In particular, for each $j=1, \dots, l$, there exist an $(n-3)$-form $\sigma^{j}$ such that $\psi^{j} = d \sigma^{j}$. 
    Furthermore, since $\psi^j$ is of type $(n-2,0)$, and $J$ is biinvariant, $\sigma^j$ is holomorphic, of type $(n-3,0)$, and $\psi^{j} = \partial \sigma^{j}$. 
    Hence, 
    \begin{equation*}
        \psi^{j} \wedge \overline{\psi^{j}} = \partial \sigma^{j} \wedge \overline{\partial \sigma^{j}} =\partial \overline{\partial} \pt{\sigma^{j} \wedge \overline{\sigma^{j}}},
    \end{equation*}
    where the last equality follows from the fact that $\sigma^{j}$ is holomorphic.
    Since each summand appearing in $\Omega$ is $\partial \overline{\partial}$-exact, we conclude that the entire form $\Omega$ is $\partial \overline{\partial}$-exact and so $[\Omega]_{BC}=0$.
\end{proof}

\smallskip
{\bf Acknowledgments.} 
The authors would like to thank Gueo Grantcharov and Elia Fusi for valuable conversations and comments. They are also grateful to Anna Fino for providing several comments that improved the clarity of the paper and to Hisashi Kasuya for his helpful remarks. The first author wishes to express his gratitude to Adriano Tomassini for his constant encouragement and support, and to Luis Ugarte for stimulating conversations. He also gratefully acknowledges the hospitality and support of the IMAR, Bucharest, and of the Institute of Mathematics and Informatics at the Bulgarian Academy of Sciences, Sofia, Bulgaria, where he was partially supported by G. Grantcharov's Simons grant (\#853269) through Florida International University, Miami, Florida USA.  
E. Lo Giudice was partly supported by the University of Parma through the action “Bando di Ateneo 2025 per la ricerca” and by GNSAGA of INdAM.
A. Mainenti was partly supported by the PNRR-III-C9-2023-I8 grant CF 149/31.07.2023 {\it Conformal Aspects of Geometry and Dynamics}.

\bibliographystyle{plain}
\bibliography{references}
\end{document}